\documentclass[final]{amsart}

\usepackage[english]{babel}
\usepackage[utf8]{inputenc}
\usepackage[T1]{fontenc}

\usepackage[all]{xy}
\usepackage{amssymb}
\usepackage{amsmath}
\usepackage{amsthm}
\usepackage{mathrsfs}
\usepackage{enumerate}
\usepackage{xspace}
\usepackage{hyperref}
\usepackage{xcolor}
\usepackage{todonotes}

\usepackage{fixme}
\usepackage[notcite,notref]{showkeys}

\newcommand{\A}{\ensuremath{\mathfrak{A}}\xspace}
\newcommand{\B}{\ensuremath{\mathfrak{B}}\xspace}
\newcommand{\Z}{\ensuremath{\mathbb{Z}}\xspace}

\newcommand{\K}{\ensuremath{\operatorname{K}}\xspace}
\newcommand{\KO}{\ensuremath{\mathbb{K}}\xspace}

\newenvironment{smallpmatrix}{%
\left(\begin{smallmatrix}}{\end{smallmatrix}\right)}
\newcommand{\PS}{\ensuremath{\mathcal{S}_2^1}\xspace}
\newcommand{\CSP}{\ensuremath{\mathcal{CSP}}\xspace}
\newcommand{\Ocal}{\ensuremath{\mathcal{O}}\xspace}
\newcommand{\FK}{\ensuremath{\operatorname{K_X}}\xspace}
\newcommand{\RFK}{\ensuremath{\operatorname{K^{red}_X}}\xspace}
\newcommand{\FKD}{\ensuremath{\operatorname{K_\mathcal{CSP}}}\xspace}
\newcommand{\RFKD}{\ensuremath{\operatorname{K^{red}_\mathcal{CSP}}}\xspace}
\newcommand{\FKS}{\ensuremath{\operatorname{K_{\PS}}}\xspace}

\newcommand{\iso}{isomorphism\xspace}

\newcommand{\shom}{\mbox{$*$-}homomorphism\xspace}

\newcommand{\cafont}[1]{\ensuremath{\mathfrak{#1}}\xspace}
\newcommand{\ca}{\mbox{$C\sp*$-}al\-ge\-bra\xspace}
\newcommand{\cas}{\mbox{$C\sp*$-}al\-ge\-bras\xspace}
\newcommand{\KXa}{Kirchberg $X$-algebra\xspace}
\newcommand{\KXas}{Kirchberg $X$-algebras\xspace}
\newcommand{\KDa}{Kirchberg $\mathcal{CSP}$-algebra\xspace}
\newcommand{\KDas}{Kirchberg $\mathcal{CSP}$-algebras\xspace}

\newcommand{\KSas}{Kirchberg $\PS$-algebras\xspace}
\newcommand{\Op}{\mathbb{O}}
\newcommand{\Ideals}{\mathbb{I}}

\newcommand{\LC}{\mathbb{LC}}
\newcommand{\into}{\rightarrowtail}
\newcommand{\onto}{\twoheadrightarrow}
\newcommand{\KK}{\ensuremath{\operatorname{KK}}\xspace}

\newcommand{\Ext}{\ensuremath{\operatorname{Ext}}\xspace}
\newcommand{\Hom}{\ensuremath{\operatorname{Hom}}\xspace}
\newcommand{\NT}{\mathcal{NT}}
\newcommand{\cirkt}{concrete ideal-related \K-theory\xspace}
\newcommand{\rcirkt}{reduced concrete ideal-related \K-theory\xspace}
\newcommand{\csa}{$C^*$-algebra}
\newcommand{\CK}{Cuntz-Krieger algebra}
\newcommand{\rrzero}{real rank zero like}
\newcommand{\Rcal}{\mathcal{R}}
\newcommand{\Rcat}{\mathcal{R}}
\DeclareMathOperator{\id}{id}

\newcommand{\pdim}{\operatorname{pd}}
\newcommand{\idim}{\operatorname{id}}

\newtheorem{theorem}{Theorem}[section]
\newtheorem{lemma}[theorem]{Lemma}
\newtheorem{corollary}[theorem]{Corollary}
\newtheorem{proposition}[theorem]{Proposition}
{\theoremstyle{definition}\newtheorem{definition}[theorem]{Definition}
\newtheorem{remark}[theorem]{Remark}
\newtheorem{observation}[theorem]{Observation}
}

\numberwithin{equation}{section}
\title[Purely infinite $C^{*}$-algebras]{Classification of real rank zero, purely infinite $C^{*}$-algebras with at most four primitive ideals} 

\author{Sara E. Arklint \and Gunnar Restorff \and Efren Ruiz}
\address{Department of Mathematical Sciences, University of Copenhagen, Universitets\-parken~5, DK-2100~Copen\-hagen~Ø, Denmark}
\email{arklint@math.ku.dk}
\address{Department of Science and Technology, University of the Faroe Islands, N\'oat\'un~3, FO-100~T\'orshavn, the Faroe Islands}
\email{gunnarr@setur.fo}
\address{Department of Mathematics, University of Hawaii, Hilo, 200~W.~Kawili St., Hilo, HI, 96720-4091, Hawaii, USA}
\email{ruize@hawaii.edu}
\begin{document}

\maketitle

\begin{abstract}
Counterexamples to classification of purely infinite, nuclear, separable \cas  (in the ideal-related bootstrap class) and with primitive ideal space $X$ using ideal-related \K-theory occur for infinitely many finite primitive ideal spaces $X$, the smallest of which having four points.
Ideal-related \K-theory is known to be strongly complete for such \cas if they have real rank zero and $X$ has at most four points for all but two exceptional spaces: the pseudo-circle and the diamond space.
In this article, we close these two remaining cases.
We show that ideal-related \K-theory is strongly complete for real rank zero, purely infinite, nuclear, separable \cas that have  the pseudo-circle as primitive ideal space. 
In the opposite direction, we construct a \CK{} with the diamond space as its primitive ideal space for which an automorphism on ideal-related \K-theory does not lift.
\end{abstract}

\section{Introduction}
The $\K$-theoretical classification of simple, purely infinite, nuclear, separable \cas---the Kirchberg algebras---in the UCT class was achieved in the early 1990s by E. Kirchberg and N.C. Phillips, \cite{kirchberg1994,phillips2000}.
For nonsimple, $\Ocal_\infty$-absorbing, nuclear, separable \cas it depends heavily on the structure of the fixed primitive ideal space whether or not \K-theoretical classification is possible.
Apart from rather special cases, \cite{bentmann}, it remains an open question how strong assumptions are needed to ensure classification.

The classification of UCT Kirchberg algebras was preceeded by M. R\o{}rdam's classification of the simple \CK{}s, \cite{MR1340839,CK80}.
Following this lead, an approach to the classification of purely infinite, nuclear, separable \cas with finitely many ideals is to use the \CK{}s as a guideline.

The classification functor in question is the ideal-related \K-theory, encompassing the six-term exact sequences in \K-theory induced by all extensions of subquotients. This functor is known to classify the purely infinite \CK{}s up to stable isomorphism and independent of the structure of the primitive ideal space, \cite{MR2270572}.
However, this work left open the question of strong completeness.
The approach is now to understand which properties of the \CK{}s that enable its completeness and determine whether or not strong completeness is possible.

In this article, we achieve this for the remaining four-point primitive ideal spaces, namely the two spaces that can be represented by the graphs
$$\vcenter
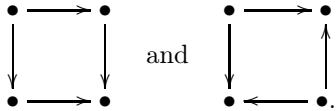
{\xymatrix{
\bullet\ar[r]\ar[d] & \bullet\ar[d] \\
\bullet\ar[r] & \bullet
}}\quad\text{and}\quad
\vcenter{\xymatrix{
\bullet\ar[r]\ar[d] & \bullet \\
\bullet & \bullet . \ar[l]\ar[u]
}}$$
After an introduction of necessary notation and definitions in~\S\ref{sec:pre}, 
the main results and a more detailed account of the chronological development are given in~\S\ref{sec:overview}.

\section{Prerequisites} \label{sec:pre}

In this section, we briefly introduce notation needed for stating the results, as well as it will be used in later sections to prove some of the results.




\subsection{\cas over a topological space $X$}

\cas over a topological space $X$ is considered in \cite{MR2953205} (as well as many other places). 
As we will only be considering 
 \cas over a finite $T_0$ space $X$, we will not define it in its full generality, but refer the reader to \cite{MR2953205} and the references therein for a more general approach. 
Assume from now on that $X$ is a finite topological space satisfying the $T_0$ separation axiom, i.e., such that $\overline{\{x\}}\neq\overline{\{y\}}$ for all $x,y\in X$ with $x\neq y$.
Let $\Op(X)$ denote the open subsets of $X$, and let $\Ideals(\A)$ denote the lattice of (two-sided, closed) ideals of a \ca \A. 
A (tight) \ca over $X$ can then 
be defined as a pair $(A,\psi)$ consisting of a \ca \A and a lattice (iso)morphism 
$\psi\colon\Op(X)\to\Ideals(\A)$
that preserves infima and suprema.
We write $\A(U)$ for $\psi(U)$. 

The locally closed subsets of $X$ are denoted by $\LC(X)=\{ U\setminus V \mid  V,U\in\Op(X), V\subseteq U\}$, and the connected, non-empty, locally closed subsets of $X$ are denoted by $\LC(X)^*$.
For a \ca \A over $X$ and for $Y\in\LC(X)$ we define $\A(Y)=\A(U)/\A(V)$ when $Y=U\setminus V$ for some $V,U\in\Op(X)$ satisfying $V\subseteq U$.  Up to canonical isomorphism, $\A(Y)$ does not depend on the choice of $U$ and $V$.

For \cas \A and \B over $X$, we say that a \shom $\phi\colon \A\to \B$ is $X$-equivariant if $\phi(\A(U))\subseteq \B(U)$ holds for all $U\in\Op(X)$.

A Kirchberg $X$-algebra is a nuclear, separable, tight \ca \A over $X$ that absorbs $\Ocal_{\infty}$, i.e., $\A\otimes\Ocal_\infty\cong\A$.  
For nuclear, separable \cas with finitely many ideals, $\Ocal_\infty$-absorption, pure infiniteness, and strong pure infiniteness are equivalent---this follows from results of Kirchberg, R\o{}rdam, Toms, and Winter (cf.~\cite[Corollary~1]{ERR:gjogv}). 

\subsection{Concrete ideal-related \K-theory}

Let $X$ be a finite $T_0$-space.
For a \ca \A over $X$, we let $\FK(\A)$ be the collection $(\K_i(\A(Y)))_{i=0,1,Y\in\LC(X)^*}$ of abelian groups together with the group homomorphisms in the cyclic six term exact sequence 
$$\xymatrix{
\K_0(\A(Y_1))\ar[r] & \K_0(\A(Y_2))\ar[r] & \K_0(\A(Y_2\setminus Y_1))\ar[d] \\
\K_1(\A(Y_2\setminus Y_1))\ar[u] & \K_1(\A(Y_2))\ar[l] & \K_1(\A(Y_1))\ar[l] 
}$$
induced by the short exact sequence $\A(Y_1)\into\A(Y_2)\onto\A(Y_2\setminus Y_1)$ whenever $Y_1\subsetneq Y_2$ where $Y_1,Y_2\in\LC(X)^*$. 
For \cas \A and \B over $X$, a homomorphism from $\FK(\A)$ to $\FK(\B)$ is a collection of group homomorphisms $\phi_{i,Y}$ from $\K_i(\A(Y))$ to $\K_i(\B(Y))$ for $i=0,1$, $Y\in\LC(X)^*$ that commute with all the relevant homomorphisms from the above cyclic six term sequences. 
In the obvious way, this defines a functor $\FK$ on the category of \cas over $X$. 

We call the functor \FK the \cirkt. In \cite{arXiv:1301.7223v2} it is called concrete filtered \K-theory, and it is known to be the same as filtrated \K-theory introduced by R.~Meyer and R.~Nest in \cite{MR2953205} for certain cases including all $T_0$-spaces with four or fewer points. 

In~\cite{MR2953205}, R.~Meyer and R.~Nest establish a UCT for $\KK_X$-theory 
and they define the bootstrap class $\mathcal{B}(X)$ in \cite[Definition~4.11]{MR2545613}. 
By \cite[Corollary~4.13]{MR2545613}, a (tight) nuclear, separable \ca \A over $X$ belongs to $\mathcal{B}(X)$ if and only if its (simple) subquotients $\A(x)$, $x\in X$, belong to the bootstrap class of J.~Rosenberg and C.~Schochet. 

\subsection{Reduced concrete ideal-related \K-theory}
Let $X$ be a finite $T_0$-space.
For $x,y\in X$, we write $y\rightarrow x$ whenever $x\neq y$, $\overline{\{x\}}\subseteq\overline{\{y\}}$ and $\overline{\{x\}}\subseteq\overline{\{z\}}\subseteq\overline{\{y\}}$ implies that $x=z$ or $y=z$. 
Moreover, we let $\widetilde{\{x\}}$ denote the smallest open set containing $x$, and we let $\widetilde{\partial}x=\widetilde{\{x\}}\setminus\{x\}$. 

\begin{definition}[{\cite[Definition~3.1]{arXiv:1309.7162v1}}]
Let $\Rcat$ denote the universal pre-additive category generated by
objects $x_1$, $\widetilde{\partial} x_0$, $\widetilde{x}_0$ for all $x\in X$ and morphisms $\delta_{x_1}^{\widetilde{\partial} x_0}$ from $x_1$ to ${\widetilde{\partial} x_0}$ and $i_{{\widetilde{\partial} x_0}}^{\widetilde{x}_0}$ from ${\widetilde{\partial} x_0}$ to $\widetilde{x}_0$ for all $x\in X$, and $i^{{\widetilde{\partial} x_0}}_{\widetilde{y}_0}$ from $\widetilde{y}_0$ to ${\widetilde{\partial} x_0}$ whenever $y\rightarrow x$, subject to the relations
$$i_{\widetilde{\partial}x_0}^{\widetilde{x}_0} \delta_{x_1}^{\widetilde{\partial}x_0} =0,\text{ for all }x\in X$$
and
$$
i_{(\widetilde{z_2})_0}^{\widetilde{\partial}(z_1)_0}i^{(\widetilde{z_{2}})_0}_{\widetilde{\partial}(z_{2})_0} i_{(\widetilde{z_3})_0}^{\widetilde{\partial}(z_{2})_0} \cdots i^{(\widetilde{z_{n-1}})_0}_{\widetilde{\partial}(z_{n-1})_0} i_{(\widetilde{z_n})_0}^{\widetilde{\partial}(z_{n-1})_0}
= i_{(\widetilde{z_2'})_0}^{\widetilde{\partial}(z_1')_0} i^{(\widetilde{z_{2}'})_0}_{\widetilde{\partial}(z_{2}')_0}
 i_{(\widetilde{z_3'})_0}^{\widetilde{\partial}(z_{2}')_0} \cdots i^{(\widetilde{z_{n-1}'})_0}_{\widetilde{\partial}(z_{n-1}')_0} i_{(\widetilde{z_n'})_0}^{\widetilde{\partial}(z_{n-1}')_0}$$
whenever $y= z_n \rightarrow  \cdots \rightarrow z_1 = x$ and $y= z_n' \rightarrow  \cdots \rightarrow z_1' = x$. 
\end{definition}
Here we choose the convention of composition $fg$ of functions $f$ and $g$ as the map that sends $t$ to $f( g(t) )$.  This is the opposite convention made in \cite[Definition~3.1]{arXiv:1309.7162v1}, where the authors write the composition $fg$ as the map that sends $t$ to $g( f(t) )$.

\begin{definition}
For a  \ca \A over $X$, we let $\RFK(\A)$ be the collection 
$$\left(\K_0(\A(\widetilde{\{x\}})),\K_0(\A(\widetilde{\partial}x)),\K_1(\A(x))\right)_{x\in X}$$ 
of abelian groups together with the group homomorphisms $\K_1(\A(x))\rightarrow\K_0(\A(\widetilde{\partial}x))\rightarrow\K_0(\A(\widetilde{\{x\}}))$ 
in the cyclic six term exact sequence induced by the short exact sequence $\A(\widetilde{\partial}x)\into\A(\widetilde{\{x\}})\onto\A(x)$ for all $x\in X$ and the 
group homomorphisms $\K_0(\A(\widetilde{\{y\}}))\rightarrow\K_0(\A(\widetilde{\partial}x))$
in the cyclic six term exact sequence induced by the short exact sequence $\A(\widetilde{\{y\}})\into\A(\widetilde{\partial}x)\onto\A(\widetilde{\partial}x\setminus \widetilde{\{y\}})$ whenever $y\rightarrow x$. 

A homomorphism from $\RFK(\A)$ to $\RFK(\B)$ is a collection of group homomorphisms 
$$\phi_{\widetilde{x}_0}\colon\K_0(\A(\widetilde{\{x\}}))\rightarrow \K_0(\B(\widetilde{\{x\}}))$$
$$\phi_{\widetilde{\partial}x_0}\colon\K_0(\A(\widetilde{\partial}x))\rightarrow\K_0(\B(\widetilde{\partial}x))$$
$$\phi_{x_1}\colon\K_1(\A(x))\rightarrow\K_1(\B(x))$$ 
for all $x\in X$ that commute with all the above maps. 
It is easy to verify that this in the natural way defines a functor $\RFK$ from the category of \cas over $X$ to the category of modules over $\Rcat$. 

We call the functor \RFK the \rcirkt. This functor was originally introduced in \cite{MR2270572} as reduced filtered \K-theory inspired by the reduced \K-web introduced by M.~Boyle and D.~Huang in \cite{MR1990568}. This reformulation is from \cite{arXiv:1309.7162v1}, where it also is called the reduced filtered \K-theory. 
\end{definition}

\begin{definition}[{\cite[Definition~3.5]{arXiv:1309.7162v1}}]
Let for an element $x$ in $X$, $DP(x)$ denote the set of pairs of distinct
paths $(p, q)$ in $X$ to $x$ and from some common element which is denoted $s(p, q)$. 
An $\Rcat$-module $M$ is called exact if the sequences
$$M(x_1) \rightarrow M(\widetilde{\partial}x_0 )\rightarrow M (\widetilde{x}_0 )$$
and 
$$\bigoplus_{(p,q)\in DP(x)} M (\widetilde{s(p, q)}_0 )
\rightarrow \bigoplus_{y\rightarrow x}M(\widetilde{y}_0)\rightarrow M(\widetilde{\partial}x_0)\rightarrow 0$$
are exact for all $x$ in $X$. 
\end{definition}
The notion of exactness of an $\Rcat$-module is inspired by $\RFK(\A)$ being exact for all real rank zero \cas \A over $X$,
\cite[Corollary~3.10]{arXiv:1309.7162v1}.

Since $\K_0(\A(X))$ is part of $\FK(\A)$, it is clear for the \cirkt what it should mean that a morphism preserves the class of the unit. For the \rcirkt, however, it is not generally the case, that $\K_0(\A(X))$ is part of the invariant. 
For a \ca \A over $X$ of real rank zero, we have an exact sequence (cf.~\cite[Lemma~5.2]{arXiv:1309.7162v1})
$$\bigoplus_{y\rightarrow x,y\rightarrow x' \\ x,x'\in X}\K_0(\A(\widetilde{\{y\}}))\rightarrow
\bigoplus_{x\in X}\K_0(\A(\widetilde{\{x\}}))\rightarrow 
\K_0(\A(X))\rightarrow 0.$$
Thus every morphism between the \rcirkt will induce a unique homomorphism on the level of the $\K_0$ of the algebras, and we can ask that this homomorphism preserves the class of the unit (cf.~\cite[Section~5]{arXiv:1309.7162v1}).






\subsection{The target category of \cirkt}\label{subsec:NT}

For spaces with at most four points, the functor filtered $\K$-theory defined in~\cite{MR2953205} is identical to \cirkt.
In fact, there are no known examples of finite spaces $X$ over which they do not agree.
For notational benefits, we use the definition of the category $\NT$ and the functor filtered $\K$-theory given in~\cite{arXiv:1301.7223v2}.  Let $X$ denote a finite $T_0$-space.

For \cas over $X$, we let $\Sigma^1=\Sigma$ denote suspension, i.e., tensoring by $C_0(\mathbb R)$, and let $\Sigma^0$ denote the identity functor. 

In~\cite{MR2953205}, R. Meyer and R. Nest construct for all $Y\in\LC(X)$ commutative \cas $\Rcal_Y$ over $X$ satisfying that $\KK_*(X;\Rcal_Y,-)$ is equivalent to $\K_*(-(Y))$ as functors of separable \cas over $X$.
Let $\NT$ denote the pre-additive category with objects $\LC(X)\times\{0,1\}$ and morphisms between $(Y,j)$ and $(Z,k)$ given by $\KK_0(X;\Sigma^k\Rcal_Z,\Sigma^j\Rcal_Y)$.
We call an additive functor $\NT\to\mathfrak{Ab}$ an $\NT$-module.
Equivalently, an $\NT$-module is a (left) module over the category ring of $\NT$, and we abuse notation by denoting this unital ring $\NT$.

For a \csa{} $\A$ over X, its filtered $\K$-theory is the $\NT$-module defined as $\left(\KK_0(X;\Sigma^j\Rcal_Y,\A)\right)_{(Y,j)\in\LC(X)\times\{0,1\}}$ with the transformations induced by Kasparov product from the left by $\KK_0(X;\Sigma^k\Rcal_Z,\Sigma^j\Rcal_Y)$.

 For $M$ a left or right $\NT$-module, we let $SM$ denote its suspension defined by $SM(Y,j)=M(Y,1-j)$, and extend this to $\NT$-morphisms in the obvious way.

Following \cite{MR2953205}, we define for $Y\in\LC(X)$ the (left) $\NT{}$-module $P_Y$ as $P_Y(Z,i)=\KK_0(X;\Sigma^i\Rcal_Z,\Rcal_Y)$.
If we for $Y\in\LC(X)$ let $e_Y$ and $\bar e_Y$ denote the class of the identity morphism in $\KK_0(X;\Rcal_Y,\Rcal_Y)$ resp.\ $\KK_0(X;\Sigma\Rcal_Y,\Sigma\Rcal_Y)$, it is easy to see that $e_{Y}$ and $\bar e_Y$ are idempotents and that $P_Y=\NT{}e_Y$ and $SP_Y\cong\NT\bar e_Y$.

Similarly, for $Y\in\LC(X)$, we define the \emph{right} $\NT{}$-module $Q_Y$ by $Q_Y(Z,i)=\KK_0(X;\Rcal_Y,\Sigma^i\Rcal_Z)$. 
We see that $Q_Y=e_Y\NT{}$ and $SQ_Y\cong\bar e_Y\NT$.

For notational convenience, we will, e.g., write $123$ instead of $\{1,2,3\}$, and $M(123_1)$ instead of $M(\{1,2,3\},1)$ for an $\NT$-module $M$ over a space $X$ with $1,2,3\in X$.

\section{Overview of results} \label{sec:overview}

In \cite{MR2270572}, the second named author classified all purely infinite Cuntz-Krieger algebras up to stable isomorphism using the \rcirkt (and using \cirkt, of course) --- this extended results of M.~Rørdam and D.~Huang. 
Although this in some way closed the Cuntz-Krieger algebra case, it raised many new question --- some of them were already posted in the addendum of this article. 
One natural question is whether there is strong classification for the stable case --- we say that a functor on a class of \cas (over $X$) is a strong classification functor if every isomorphism on the invariant level lifts to an ($X$-equivariant) isomorphism on the algebra level. 
Another natural question is whether the classification generalizes to \KXas in the bootstrap class $\mathcal{B}(X)$. 
A third natural question is how to obtain unital classification (for purely infinite Cuntz-Krieger algebras, or, more generally, for \KXas in the bootstrap class $\mathcal{B}(X)$). 

These three question turned out to be very closely related. 
In the papers \cite{MR2379290,MR2265044,arXiv:1301.7695v1}, it is shown that for purely infinite \cas, strong classification for the stabilized algebras usually implies strong classification for the unital case --- thus reducing the unital problem to lifting isomorphisms in the stable case. 

E.~Kirchberg has shown, that for separable, nuclear, stable, strongly purely infinite \cas, every invertible ideal-related $\KK$-element can be lifted to an isomorphism on the algebra level. 
In \cite{bonkat:phd,restorff:phd,MR2953205,arXiv:1101.5702v3} universal coefficient theorems have been provided for ideal-related \K-theory over so-called accordion spaces (a finite $T_0$-space $X$ is called an accordion space if the symmetrization of the graph of $X$ defined above is linear), thus providing a strong classification in both the stable and unital case for \KXas in the bootstrap class $\mathcal{B}(X)$, over accordion spaces (using the functor $\FK$ --- together with the class of the unit, in the unital case). 

Moreover, \cite{MR2953205,arXiv:1101.5702v3} provide counter examples to classification for \KXas in the bootstrap class $\mathcal{B}(X)$ for all finite $T_0$-spaces $X$, that are not disjoint unions of accordion spaces. 
There are, up to homeomorphism, six such $T_0$-spaces with $4$ (or fewer) points. 
In the following, it will be convenient with notation for two of these space, therefore we let
$$\mathcal{CSP}=\{1,2,3,4\}\quad\text{and}\quad\PS=\{1,2,3,4\}$$
denote the connected $T_0$-spaces with the topologies 
$$\{\emptyset,\{4\},\{2,4\},\{3,4\},\{2,3,4\},\{1,2,3,4\}\}$$
and
$$\{\emptyset,\{3\},\{4\},\{3,4\},\{1,3,4\},\{2,3,4\},\{1,2,3,4\}\},$$
respectively. These can also be illustrated by the graphs 
$$\vcenter{\xymatrix{
4\ar[r]\ar[d] & 2\ar[d] \\
3\ar[r] & 1
}}\quad\text{and}\quad
\vcenter{\xymatrix{
4\ar[r]\ar[d] & 1 \\
2 & 3\ar[l]\ar[u].
}},$$
respectively. 
In \cite{MR2949216} it was shown that for all $T_0$-spaces $X$ with $4$ or fewer points, that are not homeomorphic to $\mathcal{CSP}$ or $\PS$, we have strong classification for real rank zero \KXas in the bootstrap class $\mathcal{B}(X)$. 
Because there is no refined invariant for $\PS$, this case has been open and is solved in Corollary~\ref{cor-strong-classification-pseudo-circle}. 
Thus we now have the following.

\begin{theorem}[{\cite{MR2953205}, \cite{arXiv:1101.5702v3}, \cite{MR2949216}, Corollary~\ref{cor-strong-classification-pseudo-circle}}] \label{thm-fullFK}
Let $X$ be a $T_0$-space with at most four points, that is not homeomorphic to $\mathcal{CSP}$. 
Let, moreover, \A and \B be \KXas of real rank zero in the bootstrap class $\mathcal{B}(X)$. 
\begin{enumerate}
\item If \A and \B are stable, then every isomorphism from $\FK(\A)$ to $\FK(\B)$ can be lifted to an $X$-equivariant \iso from \A to \B. 

\item If \A and \B are unital, then every isomorphism from $\FK(\A)$ to $\FK(\B)$ that preserves the class of the unit can be lifted to an $X$-equivariant \iso from \A to \B. 
\end{enumerate} 
If $X$ is a disjoint union of accordion spaces, then we can drop the assumption about real rank zero. 
\end{theorem}

As mentioned above, if $X$ is not the disjoint union of accordion space, then the condition about real rank zero in the theorem above cannot be dropped. 
Moreover, it has been shown in \cite{MR2949216}, that there is even a counterexample to classification, with a pair of stable \KDas of real rank zero in the bootstrap class $\mathcal{B}(\mathcal{CSP})$. 

\begin{theorem}[{\cite{MR2953205}, \cite{arXiv:1101.5702v3}, \cite{MR2949216}}] \label{thm-counterexamples1}
Let $X$ be a connected $T_0$-space with at most four points and assume that $X$ is not an accordion space.
Then $X$ has four points, and there exist non-isomorphic stable \KXas \A and \B in the bootstrap class $\mathcal{B}(X)$ and an isomorphism from $\FK(\A)$ to $\FK(\B)$.
The \KXas \A and \B can be chosen such that $\K_*(\A(Y))$ and $\K_*(\B(Y))$ are finitely generated for all $Y\in\Op(X)$.
If $X=\mathcal{CSP}$, then \A and \B can furthermore be chosen to have real rank zero. 
\end{theorem}

In \cite{arXiv:1301.7223v2}, it is shown that for real rank zero \cas over $X$ with free $K_1$, every isomorphism of the \rcirkt can be lifted to an isomorphism of the \cirkt in the case that $X$ is a $T_0$-space with at most $4$ points.
Also they show that the stable, real rank zero \KDas in the bootstrap class $\mathcal{B}(\mathcal{CSP})$ with free $K_1$ are strongly classified by the \rcirkt. 
If we combine this with the results mentioned above (in specific \cite{arXiv:1301.7695v1} and Theorem~\ref{thm-fullFK}), then we get the following. 

\begin{theorem}[{\cite{MR2953205}, \cite{arXiv:1101.5702v3}, \cite{MR2949216}, \cite{arXiv:1301.7223v2}, Corollary~\ref{cor-strong-classification-pseudo-circle}}] \label{thm-redFK}
Let $X$ be a $T_0$-space with at most four points. 
Let, moreover, \A and \B be \KXas of real rank zero in the bootstrap class $\mathcal{B}(X)$, and assume that $\K_1(\A(x))$ is free for all non-open points $x\in X$. 
\begin{enumerate}
\item If \A and \B are stable, then every isomorphism from $\RFK(\A)$ to $\RFK(\B)$ can be lifted to an $X$-equivariant \iso from \A to \B. 
\item If \A and \B are unital, then every isomorphism from $\RFK(\A)$ to $\RFK(\B)$ that preserves the class of the unit can be lifted to an $X$-equivariant \iso from \A to \B. 
\end{enumerate} 
\end{theorem}

\begin{corollary}
For purely infinite Cuntz-Krieger algebras and purely infinite graph algebras with at most four primitive ideals, the \rcirkt 
is a strong classification functor up to stable isomorphism. 
For purely infinite Cuntz-Krieger algebras and unital, purely infinite graph algebras with at most four primitive ideals, the \rcirkt 
together with the class of the unit is a strong classification functor up to unital isomorphism. 
\end{corollary}

This leads one to wonder whether \cirkt is a strong classification functor for stable real rank zero \KDas in the bootstrap class $\mathcal{B}(\mathcal{CSP})$ with free $K_1$. 
Quite surprisingly, this turns out not to be the case --- we can even get the counter example to be the stabilization of a purely infinite Cuntz-Krieger algebra, cf.\ Proposition~\ref{prop-non-lifting-of-automorphisms}. 

\begin{theorem}[{Proposition~\ref{prop-non-lifting-of-automorphisms}}] \label{thm-counterexamples2}
There exists a stable \KDa $\mathfrak{A}$ in the bootstrap class $\mathcal{B}(\mathcal{CSP})$ and an automorphism of $\FKD(\A)$ that can not be lifted to an $\mathcal{CSP}$-equivariant automorphism on \A.  Moreover, this \ca \A can be chosen to be the stabilization of a Cuntz-Krieger algebra.
\end{theorem}

In \cite[Theorem~4.3]{arXiv:1309.7162v1} a range result is shown for $\RFK$ for both purely infinite graph algebras as well as for Cuntz-Krieger algebras (even in the unital case, as well). 
This characterization of the range is very useful, and is an important part in our proof of Theorem~\ref{thm-counterexamples2} (cf.\ Proposition~\ref{prop-non-lifting-of-automorphisms}), as well to many other applications of graph and Cuntz-Krieger algebras. 

\begin{theorem}[{\cite[Theorem~4.3]{arXiv:1309.7162v1}}]\label{thm-range}
Let $M$ be an exact $\Rcat$-module with $M(x_1)$ free for all $x$ in $X$. 
Then there exists a countable graph $E$ such that all vertices in $E$ are regular and support
at least two cycles, the \ca $C\sp*(E)$ is tight over $X$ and $\RFK(C\sp*(E))$ isomorphic to $M$. 
By construction $C\sp*(E)$ is purely infinite.
The graph $E$ can be chosen to be finite if (and only if) $M(x_1)$ and $M(\widetilde{x}_0)$ are
finitely generated, and the rank of $M(x_1)$ coincides with the rank of the cokernel
of $i\colon M(\widetilde{\partial}x_0)\rightarrow M(\widetilde{x}_0)$, for all $x$ in $X$. 
If $E$ is chosen finite, then by construction $C\sp*(E)$ is a Cuntz–Krieger algebra.
\end{theorem}

In \cite{arklint:gjogv}, a phantom Cuntz-Krieger algebra is defined as a Cuntz-Krieger algebra that looks like a purely infinite Cuntz-Krieger algebra but is not isomorphic to a Cuntz-Krieger algebra. Because all the above classification results are external (as opposed to e.g.\ \cite{MR2270572}), we have the following. 

\begin{corollary}
There are no phantom Cuntz-Krieger algebras with four or fewer components. 
\end{corollary}







This completes in a very satisfactory way the picture of the classification for graph algebras and for Cuntz-Krieger algebras with at most $4$ primitive ideals, and quite clearly points out that the correct invariant to use is not \FK but rather \RFK. The general question about (strong) classification of purely infinite graph algebras is still open as well as the general question about strong classification of purely infinite Cuntz-Krieger algebras\footnote{There is some recent work in progress by different persons in different constellations including: Bentmann, Carlsen, Meyer, Restorff, Ruiz}.

\section{Counter example for $\mathcal{CSP}$}
In this section, we will consider the space $\mathcal{CSP}$.  

The space $\mathcal{CSP}$ is considered in \cite[Section~6.2]{bentmann:master}, and the modules over $\NT{}$ (and thus also $\FKD(\A)$) is illustrated by the diagram
$$
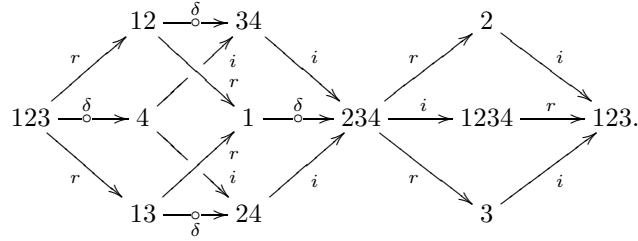
\xymatrix{
& 12 \ar[dr]^(.75)r\ar[r]|\circ^{\delta} & 34\ar[dr]^i && 2\ar[dr]^i \\ 
123\ar[ur]^r\ar[dr]_r\ar[r]|\circ^\delta & 4\ar[ur]|!{[u];[r]}\hole_(.75){i} \ar[dr]|!{[d];[r]}\hole^(.75){i} & 1\ar[r]|-\circ^-{\delta} & 234\ar[ur]^r\ar[r]^i\ar[dr]_r & 1234\ar[r]^r & 123. \\
& 13\ar[ur]_(.75)r\ar[r]|\circ_\delta & 24\ar[ur]_i & & 3\ar[ur]_i
}$$
For notational convenience, we will use subscripts to denote where a morphism starts and superscript to denote where it ends --- e.g., for a module $M$, $r_{12_0}^{1_0}$ is the homomorphism going from $M(12_0)$ to $M(1_0)$. 

Also R.~Bentmann proves in \cite[Section~6.2]{bentmann:master} a Universal Coefficient Theorem for a refined invariant $FK'$, and the $\NT{}'$-modules (and thereby also $FK'(\A)$) can be illustrated by the diagram 
$$
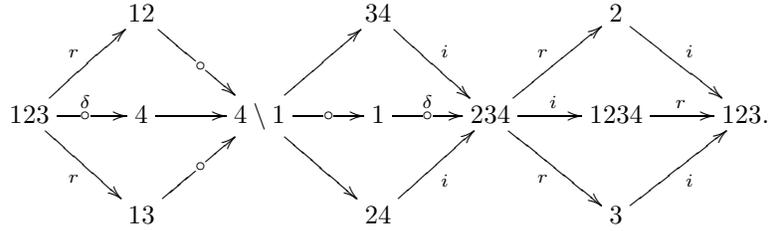
\xymatrix{
& 12 \ar[dr]|\circ & & 34\ar[dr]^i && 2\ar[dr]^i \\ 
123\ar[ur]^r\ar[dr]_r\ar[r]|\circ^\delta & 4\ar[r] & 4\setminus 1\ar[r]|-{\circ}\ar[ur]\ar[dr] & 1\ar[r]|-\circ^-{\delta} & 234\ar[ur]^r\ar[r]^i\ar[dr]_r & 1234\ar[r]^r & 123. \\
& 13\ar[ur]|\circ & & 24\ar[ur]_i & & 3\ar[ur]_i
}$$
This implies with a result of E.~Kirchberg (\cite{MR1796912}), that every isomorphism between the refined invariants of \KDas in the bootstrap class $\mathcal{B}( \mathcal{CSP} )$ can be lifted to an $\mathcal{CSP}$-equivariant isomorphism.

In \cite[Section~3.10]{bentmann:gjogv}, R.~Bentmann constructs a Cuntz-Krieger algebra $\Ocal_A$ that is tight over $\mathcal{CSP}$ such that the projective dimension of $\FKD(\Ocal_A)$ is $2$ (as an $\NT{}$-module). 
The matrix $A$ is given by
$$A=
\begin{pmatrix}
(3) & 0 & 0 & 0 \\ 
(2) & (3) & 0 & 0 \\
(2) & 0 & (3) & 0 \\
\begin{pmatrix}2 \\ 0\end{pmatrix} & \begin{pmatrix}1 \\ 0\end{pmatrix} & \begin{pmatrix}1 \\ 0\end{pmatrix} & \begin{pmatrix}2 & 1 \\ 1 & 2\end{pmatrix}
\end{pmatrix}.$$

Now we can compute the \K-theory of the Cuntz-Krieger algebra $\Ocal_A$, as described in \cite{MR2270572}. 
We let $M=\FKD(\Ocal_A)$, and get 
$$\def\arraystretch{1.0}\arraycolsep=1.4pt
\begin{array}{rlrlrlrl}
M(1_0)&=\Z,& M(2_0)&=\Z_2,& M(3_0)&=\Z_2,& M(4_0)&=\Z_2,\\
M(12_0)&=\Z,&  M(13_0)&=\Z,&  M(24_0)&=\Z_2^2,& M(34_0)&=\Z_2^2,\\
M(123_0)&=\Z_2\oplus\Z,&\quad M(234_0)&=\Z_2^3,&\quad M(1234_0)&=\Z_2^2\oplus\Z,\\
M(1_1)&=\Z,& M(2_1)&=0,& M(3_1)&=0,& M(4_1)&=0,\\
M(12_1)&=\Z,&  M(13_1)&=\Z,&  M(24_1)&=0,& M(34_1)&=0,\\
M(123_1)&=\Z,&  M(234_1)&=0,& M(1234_1)&=\Z.
\end{array}$$


\begin{lemma}\label{nonliftkirchberg}
There exists an automorphism of $N = \FKD( \Sigma \Rcal_{234} \oplus  \Ocal_{A} )$ that does not lift to an automorphism of $FK'( \Sigma \Rcal_{234} \oplus  \Ocal_{A} )$.  Moreover, we may replace $\Sigma \Rcal_{234} \oplus  \Ocal_{A}$ with a \KDa \A in the bootstrap class $\mathcal{B} ( \mathcal{CSP} )$.
\end{lemma}
\begin{proof}
We will now consider the projective module $P=SP_{234}$ from \cite[Section~6.2]{bentmann:master}. The individual groups are
$$\def\arraystretch{1.0}\arraycolsep=1.4pt
\begin{array}{rlrlrlrl}
P(1_0)&=0,& P(2_0)&=0,& P(3_0)&=0,& P(4_0)&=\Z, \\
P(12_0)&=0,& P(13_0)&=0,& P(24_0)&=0,& P(34_0)&=0, \\
P(123_0)&=0,& P(234_0)&=0,& P(1234_0)&=0, \\
P(1_1)&=0,& P(2_1)&=\Z,& P(3_1)&=\Z,& P(4_1)&=0, \\
P(12_1)&=\Z,& P(13_1)&=\Z,& P(24_1)&=0,& P(34_1)&=0,\\
P(123_1)&=\Z\oplus\Z,& \quad P(234_1)&=\Z,&\quad P(1234_1)&=\Z.
\end{array}$$


For $M$, we can choose the generators of the groups in such a way that: 
$$r_{1234_1}^{123_1}=(1), \quad \delta_{123_1}^{4_0}=(0), \quad r_{123_1}^{12_1}=(1),\quad r_{123_1}^{13_1}=(1).$$


For $P$, we can choose the generators of the groups in such a way that: 
$$r_{1234_1}^{123_1}=\begin{pmatrix}1 \\ 1\end{pmatrix},\quad  
\delta_{123_1}^{4_0}=\begin{pmatrix}-1 & 1\end{pmatrix},\quad 
r_{123_1}^{12_1}=\begin{pmatrix}1 & 0\end{pmatrix},\quad
r_{123_1}^{13_1}=\begin{pmatrix}0 & -1\end{pmatrix}.$$

Now, we let $N=P\oplus M$, and define a family of automorphisms $$\left(\alpha_{i,Y}\colon N(Y_i)\rightarrow N(Y_i)\right)_{i=0,1,Y\in\LC( \mathcal{CSP} )^*}$$ as follows. 
For all $Y\in\LC( \mathcal{CSP} )^*$, we let $\alpha_{0,Y}$ be the identity. 
For all $Y\in\{1,2,3,4,24,34,234\}$, we let $\alpha_{1,Y}$ be the identity.
Moreover, we let
\begin{align*}
\alpha_{1,1234}\colon& \Z\oplus\Z\rightarrow \Z\oplus\Z \text{ be given by the matrix }
\begin{pmatrix}1 & 1 \\ 0 & 1\end{pmatrix}, \\
\alpha_{1,123}\colon& \Z\oplus\Z\oplus\Z\rightarrow \Z\oplus\Z\oplus\Z \text{ be given by the matrix }
\begin{pmatrix}1 & 0 & 1 \\ 0 & 1 & 1 \\ 0 & 0 & 1\end{pmatrix}, \\
\alpha_{1,12}\colon& \Z\oplus\Z\rightarrow \Z\oplus\Z \text{ be given by the matrix }
\begin{pmatrix}1 & 1 \\ 0 & 1\end{pmatrix}, \\
\alpha_{1,13}\colon& \Z\oplus\Z\rightarrow \Z\oplus\Z \text{ be given by the matrix }
\begin{pmatrix}1 & -1 \\ 0 & 1\end{pmatrix}.
\end{align*}

Now we want to show that $\alpha$ is in fact an automorphism of the \cirkt. 
Commutativity of any of the squares not including any of the groups $N(12_1), N(13_1), N(123_1), N(1234_1)$ is clear, since there $\alpha$ acts as the identity. 
The equations 
\begin{align*}
\alpha_{1,1234}\circ i_{234_1}^{1234_1}&=i_{234_1}^{1234_1}\circ\alpha_{1,234}, \\
\alpha_{1,123}\circ i_{2_1}^{123_1}&=i_{2_1}^{123_1}\circ\alpha_{1,2}, \\
\alpha_{1,123}\circ i_{3_1}^{123_1}&=i_{3_1}^{123_1}\circ\alpha_{1,3}
\end{align*}
clearly hold, since the automorphisms act as the identity on the part coming from $P$. 
We have that $\alpha_{1,123}\circ r_{1234_1}^{123_1}=r_{1234_1}^{123_1}\circ\alpha_{1,1234}$, since 
$$\begin{pmatrix}1&0&1\\0&1&1\\0&0&1\end{pmatrix}\begin{pmatrix}1&0\\1&0\\0&1\end{pmatrix}=\begin{pmatrix}1&0\\1&0\\0&1\end{pmatrix}\begin{pmatrix}1&1\\0&1\end{pmatrix}.$$ 
We have that $\alpha_{0,4}\circ \delta_{123_1}^{4_0}=\delta_{123_1}^{4_0}\circ\alpha_{1,123}$, since 
$$\begin{pmatrix}1&0\\0&1\end{pmatrix}\begin{pmatrix}-1&1&0\\0&0&0\end{pmatrix}=\begin{pmatrix}-1&1&0\\0&0&0\end{pmatrix}\begin{pmatrix}1&0&1\\0&1&1\\0&0&1\end{pmatrix}.$$ 
We have that $\alpha_{1,12}\circ r_{123_1}^{12_1}=r_{123_1}^{12_1}\circ\alpha_{1,123}$, since 
$$\begin{pmatrix}1&1\\0&1\end{pmatrix}\begin{pmatrix}1&0&0\\0&0&1\end{pmatrix}=\begin{pmatrix}1&0&0\\0&0&1\end{pmatrix}\begin{pmatrix}1&0&1\\0&1&1\\0&0&1\end{pmatrix}.$$ 
We have that $\alpha_{1,13}\circ r_{123_1}^{13_1}=r_{123_1}^{13_1}\circ\alpha_{1,123}$, since 
$$\begin{pmatrix}1&-1\\0&1\end{pmatrix}\begin{pmatrix}0&-1&0\\0&0&1\end{pmatrix}=\begin{pmatrix}0&-1&0\\0&0&1\end{pmatrix}\begin{pmatrix}1&0&1\\0&1&1\\0&0&1\end{pmatrix}.$$ 

The equations 
$$\alpha_{0,34}\circ \delta_{12_1}^{34_0}=\delta_{12_1}^{34_0}\circ\alpha_{1,12},$$ 
$$\alpha_{0,24}\circ \delta_{13_1}^{24_0}=\delta_{12_1}^{24_0}\circ\alpha_{1,13},$$ 
$$\alpha_{1,1}\circ r_{12_1}^{1_1}=r_{12_1}^{1_1}\circ\alpha_{1,12},$$ 
$$\alpha_{1,1}\circ r_{13_1}^{1_1}=r_{13_1}^{1_1}\circ\alpha_{1,13},$$ 
clearly hold, since the cross terms in $\alpha_{1,12}$ and $\alpha_{1,13}$ get cancelled as the part coming from $P$ in $N(34_0)$, $N(24_0)$, and $N(1_1)$ is zero. This means that $\alpha$ is really an automorphism of the $\NT{}$-module $N=P\oplus M$. 

We know that $M=\FKD(\Ocal_A)$, $P=\FKD(\Sigma\Rcal_{234})$ and thus $N=\FKD(\Sigma\Rcal_{234}\oplus\Ocal_A)$, and thus these can be extended to modules over $\NT{}'$ as the images under the refined invariant, cf.\ \cite{bentmann:master}. With a slight misuse of notation, we also use $M$, $P$ and $N$ to denote these.
Then from \cite[Definition~6.2.10 and Lemma~6.2.11]{bentmann:master}, it follows that we have an exact sequence 
\begin{equation}
\label{eq-newgroupssequence1}
\vcenter{\xymatrix{
M(4_0)\ar[r]&M(4\setminus 1_0)\ar[r] & M(1_1)\ar[d] \\
M(1_0)\ar[u]&M(4\setminus 1_1)\ar[l] & M(4_1).\ar[l] 
}}
\end{equation}
Since the vertical homomorphisms from $M(1_i)$ to $M(4_{i})$ are the compositions of maps $M(1_i)\rightarrow M(234_{1-i})\rightarrow M(3_{1-i})\rightarrow M(123_{1-i})\rightarrow M(4_{i})$, it follows that these are both zero, and consequently, the sequence degenerates into two short exact sequences: $\Z_2\into M(4\setminus 1_0)\onto \Z$ and $0\into M(4\setminus 1_1)\onto \Z$. 
Therefore $M(4\setminus 1_0)=\Z\oplus\Z_2$ and $M(4\setminus 1_1)=\Z$. 
From \cite[Definition~6.2.10 and Lemma~6.2.11]{bentmann:master}, it also follows that we have an exact sequence 
\begin{equation}
\label{eq-newgroupssequence2}
\vcenter{\xymatrix{
M(12_1)\ar[r]&M(4\setminus 1_0)\ar[r] & M(24_0)\ar[d] \\
M(24_1)\ar[u]&M(4\setminus 1_1)\ar[l] & M(12_0).\ar[l] 
}}
\end{equation}
Since $M(24_1)=0$, $M(24_0)=\Z_2\oplus\Z_2$ and $M(12_0)=\Z$, the upper row is just a short exact sequence $\Z\into\Z\oplus\Z_2\onto\Z_2 \oplus \Z_2$. 
Therefore, we can choose the generators of $M(4\setminus 1_0)$ in such a way, that the homomorphism from $M(12_1)$ to $M(4\setminus 1_0)$ is given by the matrix $\begin{pmatrix}2\\0\end{pmatrix}$ (because if this homomorphism was of the form $\begin{pmatrix}a\\1\end{pmatrix}$ for some $a\in\Z\setminus\{0\}$ then the quotient was cyclic). 

The analogue sequence to \eqref{eq-newgroupssequence1} gives immediately, that $P(4\setminus 1_0)=\Z$ and $P(4\setminus 1_1)=0$. 
The analogue sequence to \eqref{eq-newgroupssequence2} gives immediately, that the homomorphism from $P(12_1)$ to $P(4\setminus 1_0)$ is an isomorphism---so we can choose the generators of $P(4\setminus 1_0)$ in such a way, that this isomorphism becomes multiplication by $1$. 

Since the refined invariant respects finite direct sums, we get $N(4\setminus 1_0)=\Z\oplus\Z\oplus\Z_2$ and $N(4\setminus 1_1)=0\oplus\Z$, and that the map from $N(12_1)$ to $N(4\setminus 1_0)$ is $\begin{pmatrix}1 & 0\\0&2\\0&0\end{pmatrix}$.

Now assume, that there is a homomorphism $\beta\colon N(4\setminus 1_0)\rightarrow N(4\setminus 1_0)$ that commutes with all the other maps. 
So $\beta$ is an endomorphism of $\Z\oplus\Z\oplus\Z_2$ given by an integer matrix $B$. 
If we look at the commutative diagram 
$$\xymatrix{\Z\oplus\Z\ar[r]^-{\begin{smallpmatrix}1&0\\0&2\\0&0\end{smallpmatrix}}\ar[d]_{\begin{smallpmatrix}1&1\\0&1\end{smallpmatrix}} & 
\Z\oplus\Z\oplus\Z_2\ar[d]^B \\ 
\Z\oplus\Z\ar[r]_-{\begin{smallpmatrix}1&0\\0&2\\0&0\end{smallpmatrix}} & 
\Z\oplus\Z\oplus\Z_2 }$$
we see that we need to have that 
\[
B\begin{pmatrix}0\\2\\0\end{pmatrix}
=B\begin{pmatrix}1&0\\0&2\\0&0\end{pmatrix}\begin{pmatrix}0\\1\end{pmatrix}
=\begin{pmatrix}1&0\\0&2\\0&0\end{pmatrix}\begin{pmatrix}1&1\\0&1\end{pmatrix}\begin{pmatrix}0\\1\end{pmatrix}
=\begin{pmatrix}1 \\ 2 \\ 0\end{pmatrix},
\]
which is a contradiction to $B$ being an integer matrix. Thus the above together with \cite[Theorem~5.3]{MR2545613}, proves the lemma.
\end{proof}

\begin{lemma}\label{directsum}
Let \A be a \ca over $\mathcal{CSP}$.  Suppose there exists an automorphism $\alpha$ of $\FKD ( \A )$ that does not extend to an automorphism of $FK'(\A)$.  Then for any \ca \B over $\CSP$, the automorphism $\alpha \oplus \mathrm{id}_{ \FKD ( \B ) }$ on $\FKD ( \A \oplus \B )$ does not extend to an automorphism of $FK'( \A \oplus \B )$.

Consequently, for any \ca $\cafont{C}$ over $\CSP$ which is $\KK_{\mathcal{CSP}}$-equivalent to $\A \oplus \B$, there exists an automorphism of $\FKD ( \cafont{C} )$ that does not extend to an automorphism of $FK'( \cafont{C} )$.
\end{lemma}

\begin{proof}
Let $\pi \colon  \A \oplus \B \rightarrow \A$ be the projection onto the 1st coordinate and let $\iota \colon \A \rightarrow \A \oplus \B$ be the inclusion into the 1st coordinate.  These $\mathcal{CSP}$-equivariant homomorphisms then induce surjective group homomorphisms $\lambda_{1} \colon \text{End}_{\NT{}'} ( FK'( \A \oplus \B ) ) \rightarrow \text{End}_{\NT{}'} ( FK'( \A  ) )$ and $\lambda_{2} \colon \text{End} ( \FKD( \A \oplus \B ) ) \rightarrow \text{End} ( \FKD( \A  ) )$.

Suppose there exists an endomorphism $\gamma$ of $FK'( \A \oplus \B )$ that extends $\alpha \oplus \mathrm{id}_{ \FKD ( \B ) }$.  Since there is a natural transformation from the functor $FK'$ to $\FKD$, we have group homomorphisms $\beta_{1} \colon \text{End}_{\NT{}'} ( FK'( \A \oplus \B ) ) \rightarrow \text{End} ( \FKD( \A \oplus \B ) )$ and $\beta_{2} \colon \text{End}_{\NT{}'} ( FK'( \A ) ) \rightarrow \text{End} ( \FKD( \A  ) )$ such that the diagram 
\begin{align*}
\xymatrix{
\text{End}_{\NT{}'} ( FK'( \A \oplus \B ) )  \ar[r]^-{ \lambda_{1} } \ar[d]_{ \beta_{1} } & \text{End}_{\NT{}'} ( FK'( \A ) ) \ar[d]^{ \beta_{2} } \\
\text{End} ( \FKD( \A \oplus \B ) ) \ar[r]_-{\lambda_{2} } & \text{End} ( \FKD( \A  ) )
}
\end{align*}
is commutative.  Note that $\beta_{1} ( \gamma ) =  \alpha \oplus \mathrm{id}_{ \FKD ( \B ) }$ and $\lambda_{2} ( \alpha \oplus \mathrm{id}_{ \FKD ( \B ) } ) = \alpha$.  By the commutativity of the diagram, we have that $\lambda_{1} ( \gamma )$ is an endomorphism of $FK'( \A )$ that extends $\alpha$.  Since $\alpha$ is an automorphism of $\FKD ( \A )$, by the Five Lemma, we have that $\lambda_{1} ( \gamma )$ is an automorphism of $FK'( \A )$.  Thus, $\lambda_{1} ( \gamma )$ is an automorphism of $FK'( \A )$ that extends $\alpha$ which is a contradiction.

For the last part of the lemma, if $y$ is invertible in $\KK_\mathcal{CSP}(\cafont{C} , \A \oplus \B )$, then $\FKD (y) ( \alpha \oplus \mathrm{id}_{ \FKD ( \B ) } ) \FKD (y)^{-1}$ is an automorphism of $\FKD ( \cafont{C} )$ that does not extend to an automorphism of $FK'( \cafont{C} )$.
\end{proof}

\begin{proposition}\label{prop-non-lifting-of-automorphisms}
There exists a Cuntz-Krieger algebra \A that is a tight \ca over $\mathcal{CSP}$ and for which there exists an automorphism of $\FKD ( \A )$ that does not lift to an automorphism of $\A \otimes \KO$.
\end{proposition}

\begin{proof}
By Lemma~\ref{nonliftkirchberg}, there exists an automorphism $\beta$ of $\FKD ( \Sigma R_{234} \oplus \Ocal_{A} ) = N$ that does not extend to an automorphism of $FK' (\Sigma R_{234} \oplus \Ocal_{A} )$.  Let $\cafont{C} = \Sigma R_{234} \oplus \mathcal{O}_{A} \oplus R_{24} \oplus R_{34} \oplus \Sigma R_{4}$.  A computation shows that $\mathrm{rank}( \K_{1} ( \cafont{C} ( x ) ) ) = \mathrm{rank}( \K_{0} ( \cafont{C} (x) ) )$ for all $x\in\mathcal{CSP} $.  Since $\cafont{C}$ is in the bootstrap category $\mathcal{B} ( \mathcal{CSP} )$, by \cite[Theorem~5.3]{MR2545613}, there exists a Kirchberg $\mathcal{CSP}$-algebra $\cafont{D}$ in $\mathcal{B}(\mathcal{CSP})$ such that $\cafont{D}$ is $\KK_{\mathcal{CSP}}$-equivalent to $\cafont{C}$.  By Theorem~\ref{thm-range}, 
there exists a Cuntz-Krieger algebra \A that is a tight \ca over $\mathcal{CSP}$ such that $\RFKD ( \A ) \cong \RFKD( \cafont{D} )$.  By Theorem~\ref{thm-range}, 
$\cafont{D} \otimes \KO \cong \A \otimes \KO$.  In particular, \A is $\KK_{\mathcal{CSP}}$-equivalent to $\cafont{C}$.  By Lemma~\ref{directsum}, there exists an automorphism $\alpha$ of $\FKD ( \A )$ that does not extend to an automorphism $FK'( \A )$.  This implies that $\alpha$ does not lift to an automorphism of $\A \otimes \KO$.
\end{proof}

\begin{remark}
By results in \cite{arXiv:1309.7162v1} and \cite{arXiv:1301.7223v2}, the \KDa \A appearing in Lemma~\ref{nonliftkirchberg} is unique up to stable isomorphism and can be chosen to be a graph \ca.
By the same results, the \CK{} \A appearing in Proposition~\ref{prop-non-lifting-of-automorphisms} is unique up to stable isomorphism and can be chosen to be represented by a $8\times 8$ matrix over the non-negative integers.
\end{remark}

\section{Preliminary homological algebra}
The following definition is taken from \cite[\S{}I.8]{MR1438546}. 

\begin{definition}
  \label{def-functorialiso1}
  Let $R$ be a non-trivial ring with identity. 
  Let $M$ be a right $R$-module and let $G$ be an abelian group. 
  Then we can equip the abelian group $\Hom_\Z(M,G)$ with a (left)
  $R$-module structure as follows: 
  $$(r\varphi)(x)=\varphi(x r),\quad 
  x\in M,r\in R,\varphi\in\Hom_\Z(M,G).$$
  It is an easy exercise to verify that this makes $\Hom_\Z(M,G)$ into
  a (left) $R$-module and $\Hom_\Z(-,-)$ into a bifunctor from the category of right $R$-modules and the category of abelian groups to the category of $R$-modules.
\end{definition}

\begin{definition}
  \label{def-functorialiso2}
  Let $R$ be a non-trivial ring with identity. 
  Let $M$ be an $R$-module and let $G$ be an abelian group. 
  Then $M\otimes_\Z G$ is clearly a $\Z$-module. 
  But for each $r\in R$ we can uniquely define a $\Z$-module homomorphism by 
  $$f_r\colon M\otimes_\Z G\ni x\otimes g\mapsto r x\otimes g\in M\otimes_\Z G.$$
  By uniqueness we see, that for all $r_1,r_2\in R$
  we have 
  $f_{r_1+r_2}=f_{r_1}+f_{r_2}$, 
  $f_{r_1r_2}=f_{r_1}\circ f_{r_2}$, 
  $f_1=\operatorname{id}_{M\otimes_\Z G}$. 
  Consequently, the left action of $R$ on $M\otimes_\Z G$ given
  by 
  $R\times(M\otimes_\Z G)\ni(r,x)\mapsto f_r(x)\in M\otimes_\Z G$
  makes $M\otimes_\Z G$ into an $R$-module and $-\otimes_\Z-$ into a bifunctor from the category of $R$-modules and the category of abelian groups to the category of $R$-modules.
\end{definition}

Inspired by \cite[Proposition~I.8.1]{MR1438546}, we prove the
following two propositions. 

\begin{proposition}
  \label{prop-functorialiso1}
  Let $R$ be a non-trivial ring with identity, and let $e$ be a non-zero
  idempotent in $R$. 
  Regard $eR$ as a right $R$-module. 
  Let $M$ be an $R$-module, and let $G$ be an abelian group. 
  Regard $\Hom_\Z(eR,G)$ as an $R$-module as
  above. 
  Then we have a functorial isomorphism
  $$\eta_M\colon\Hom_R(M,\Hom_\Z(eR,G))\rightarrow\Hom_\Z(eM,G)$$
  of abelian groups. 
\end{proposition} 
\begin{proof}
  The proof of this proposition is quite similar to the proof of
  \cite[Proposition~I.8.1]{MR1438546} (where $e=1$). 
  For each $\varphi\in\Hom_R(M,\Hom_\Z(eR,G))$, 
  we define $\eta_M(\varphi)=\varphi'\in\Hom_\Z(eM,G)$ by 
  $$\varphi'(x)=(\varphi(x))(e),\quad x\in eM.$$

  For each $\psi\in\Hom_\Z(eM,G)$, 
  we define a map $\psi'\colon M\rightarrow\Hom_\Z(eR,G)$ by 
  $$(\psi'(x))(r)=\psi(r x),\quad r\in eR,x\in M$$
  (clearly, $\psi'(x)$ is a $\Z$-module homomorphism). 
  Clearly, $\psi'$ is a $\Z$-module homomorphism, and for $r\in R$, $r'\in eR$ and
  $x\in M$ we
  have 
  $$\psi'(r x)(r')=\psi(r' r x)= \psi'(x)(r'r)
  =(r(\psi'(x)))(r').$$
  Hence, $\psi'$ is an $R$-module homomorphism. 

  It is evident, that $\varphi\mapsto\varphi'$ and $\psi\mapsto\psi'$
  are $\Z$-module homomorphisms. 
  Moreover, $(\varphi')'=\varphi$ and $(\psi')'=\psi$. 
  First, we see immediately that
  $(\psi')'(x)=(\psi'(x))(e)=\psi(ex)=\psi(x)$ for all $x\in eM$. 
  Moreover, for $x\in M$ and $r\in eR$ we have 
  $$((\varphi')'(x))(r)=\varphi'(rx)
  =(\varphi(rx))(e)=(r(\varphi(x)))(e)=(\varphi(x))(er)=(\varphi(x))(r).$$

  Thus $\eta_M=[\varphi\mapsto\varphi']$ is a $\Z$-module
  isomorphism, with inverse $\psi\mapsto\psi'$. 
  Functoriality is straightforward to check.
\end{proof}

\begin{proposition}
  \label{prop-functorialiso2}
  Let $R$ be a non-trivial ring with identity, and let $e$ be a non-zero
  idempotent in $R$. 
  Let $M$ be an $R$-module, and let $G$ be an abelian group. 
  Regard $Re\otimes_\Z G$ as an $R$-module as
  above. 
  Then we have a functorial isomorphism
  $$\eta_M\colon\Hom_R(Re\otimes_\Z G,M)\rightarrow\Hom_\Z(G,eM)$$
  of abelian groups. 
\end{proposition}
\begin{proof}
  For each $\varphi\in\Hom_R(Re\otimes_\Z G,M)$, we define 
  $\varphi'\in\Hom_\Z(G,eM)$ by
  $$\varphi'(g)=\varphi(e\otimes g)=e\varphi(1\otimes g),
  \quad g\in G.$$
  
  Using the universal property for tensor products, for each
  $\psi\in\Hom_\Z(G,eM)$, we define a unique $\Z$-module
  homomorphism $\psi'\colon Re\otimes_\Z G\rightarrow M$ by 
  $$\psi'(r\otimes g)=r\psi(g),\quad \text{for all }r\in Re,g\in G.$$
  It is straightforward to check that $\psi'$ is an $R$-module
  homomorphism. 
  
  Clearly, $\varphi\mapsto\varphi'$ and $\psi\mapsto\psi'$ are
  $\Z$-module homomorphisms. Moreover, $(\varphi')'=\varphi$ and
  $(\psi')'=\psi$:
  $$(\varphi')'(r\otimes g)=r\varphi'(g)
  =r\varphi(e\otimes g)fuzz
  =\varphi(r(e\otimes g))
  =\varphi(r\otimes g),\quad r\in Re,g\in G,$$
  $$(\psi')'(g)=\psi'(e\otimes g)=e\psi(g)=\psi(g),
  \quad g\in G.$$

  Thus $\eta_M=[\varphi\mapsto\varphi']$ is a $\Z$-module
  isomorphism, with inverse $\psi\mapsto\psi'$. 
  Functoriality is straightforward to check. 
\end{proof}

\begin{remark}\label{remark-proj-homo}
  Let $R$ be a non-trivial ring with identity, let $e\in R$ be a non-zero
  idempotent, and let $G$ be an abelian group. 
  Let, moreover, $M$ be an $R$-module, and let 
  $\phi\colon G\rightarrow e M$ be a $\Z$-module homomorphism.  
  Then it follows from Proposition~\ref{prop-functorialiso2} (and its proof) that there exists exactly one
  $R$-module homomorphism 
  $\phi_{G,e}\colon Re\otimes_\Z G\rightarrow M$ such that 
  $$\phi_{G,e}(e\otimes g)=\phi(g),
  \quad\text{for all }g\in G.$$

  It is immediate from the proof of Proposition~\ref{prop-functorialiso2}, that $\phi_{G,e}|_{e(Re\otimes_\Z G)}$ surjects onto $eM$ whenever $\phi$ is surjective, and that $\phi$ is injective whenever $\phi_{G,e}|_{e(Re\otimes_\Z G)}$ is injective. The converse statements do not hold in general.
However, if $eRe=\Z e$ is a free group in one generator, then the converse statements also hold.
\end{remark}

\begin{remark}\label{remark-inj-homo}
  Let $R$ be a non-trivial ring with identity, let $e\in R$ be a non-zero
  idempotent, and let $G$ be an abelian group. 
  Let, moreover, $M$ be a $R$-module, and let 
  $\phi\colon eM \rightarrow G$ be a $\Z$-module homomorphism.  
  Then it follows from Proposition~\ref{prop-functorialiso1} (and its proof) that there exists exactly one
  $R$-module homomorphism 
  $\phi_{G,e}\colon M\rightarrow\Hom_\Z(eR, G)$
  such that 
  $$(\phi_{G,e}(x))(e)=\phi(e x), 
  \quad\text{for all }x\in M.$$  
  Surely, the existence is clear from this proposition (and its proof), but also the
  uniqueness is clear, since for $\phi_{G,e}$ satisfying this we have for all
  $x\in M$ and $r\in R$ that 
  \begin{align*}
    (\phi_{G,e}(x))(er)
    =(r(\phi_{G,e}(x)))(e) 
    =(\phi_{G,e}(r x))(e) 
    =\phi(er x). 
  \end{align*}

  It is immediate from the proof of Proposition~\ref{prop-functorialiso1}, that 
  $\phi_{G,e}|_{e\Hom_\Z(eR,G)}$ is injective whenever $\phi$ is
  injective, and that $\phi$ is surjective whenever
  $\phi_{G,e}|_{e\Hom_\Z(eR,G)}$ is surjective.
  The converse statements do not hold in general. However, if $eRe=\Z e$ is a free group in one generator, then the converse statements also hold.
\end{remark}

\begin{proposition}\label{prop-proj-inj}
  Let $R$ be a non-trivial ring with identity, let $e\in R$ be a non-zero
  idempotent.
  For each projective $\Z$-module $P$ (i.e., free abelian group), 
  the $R$-module $Re\otimes_\Z P$
  is projective. 
  For each injective $\Z$-module $I$ (i.e., divisible abelian group), 
  the $R$-module $\Hom_\Z(eR,I)$ 
  is injective. 
\end{proposition}
\begin{proof}
  If $0\rightarrow L\rightarrow M\rightarrow N\rightarrow 0$
  is a short exact sequence of $R$-modules, 
  the induced sequence 
  $0\rightarrow eL\rightarrow eM\rightarrow eN\rightarrow 0$
  abelian groups is also short exact. 
  Since $\Hom_\Z(P,-)$ is exact when $P$ is projective, and  
  $\Hom_\Z(-,I)$ is exact when $I$ is injective, the results follow from
  the functorial isomorphisms of
  Propositions~\ref{prop-functorialiso1} and~\ref{prop-functorialiso2}. 
\end{proof}

\begin{proposition}\label{prop-proj-inj-dim}
  Let $R$ be a non-trivial ring with identity, let $e\in R$ be a non-zero
  idempotent.
  If $Re$ is flat as a \Z-module (i.e., torsion-free abelian group), then for each $\Z$-module $G$, 
  the $R$-module $Re\otimes_\Z G$
  has projective dimension at most $1$. 
  If $eR$ is projective as a \Z-module, then for each $\Z$-module $G$, 
  the $R$-module $\Hom_\Z(eR,G)$ 
  has injective dimension at most $1$. 
\end{proposition}
\begin{proof}
Assume that $Re$ is flat as a \Z-module, and let 
$P_1\into P_0\onto G$ be a projective resolution of $G$. 
Then $Re\otimes_\Z -$ is exact, so we have a short exact sequence of $\Z$-modules 
$$Re\otimes P_1\into Re\otimes P_0\onto Re\otimes G.$$
It is straightforward to show that this is indeed a short exact sequence of $R$-modules, and thus it follows from Proposition~\ref{prop-proj-inj} that this is a projective resolution. 

Assume that $eR$ is projective as a \Z-module, and let 
$G\into I_0\onto I_1$ be an injective resolution of $G$. 
Then $\Hom_\Z(eR,-)$ is exact, so we have a short exact sequence of $\Z$-modules 
$$\Hom_\Z(eR,G)\into \Hom_\Z(eR,I_0)\onto \Hom_\Z(eR,I_1).$$
It is straightforward to show that this is indeed a short exact sequence of $R$-modules, and thus it follows from Proposition~\ref{prop-proj-inj} that this is an injective resolution. 
\end{proof}

\begin{lemma} \label{lem-reduction-ext-2}
Let $R$ be a non-trivial ring with identity.
Let $K\into L\onto M$ be a short exact sequence of $R$-modules, and let $N$ be an $R$-module. 
Then the following holds: 
$$\Ext_R^2(K,N)=0\wedge\Ext_R^2(M,N)=0\Longrightarrow \Ext_R^2(L,N)=0,$$
$$\Ext_R^2(N,K)=0\wedge\Ext_R^2(N,M)=0\Longrightarrow \Ext_R^2(N,L)=0.$$
Consequently, 
$$\pdim_R K\leq 1\wedge \Ext_R^2(M,N)=0\Longrightarrow \Ext_R^2(L,N)=0,$$
$$\pdim_R M\leq 1\wedge \Ext_R^2(K,N)=0\Longrightarrow \Ext_R^2(L,N)=0,$$
$$\idim_R K\leq 1\wedge \Ext_R^2(N,M)=0\Longrightarrow \Ext_R^2(N,L)=0,$$
$$\idim_R M\leq 1\wedge \Ext_R^2(N,K)=0\Longrightarrow \Ext_R^2(N,L)=0.$$
\end{lemma}

\begin{proof}
This is a direct consequence of the long exact sequences for derived functors.
\end{proof}

\begin{lemma} \label{lemma:hom-tensor-iso}
Let $R$ be a non-trivial ring with identity.  Let $M$ be a right $R$-module and let $G$ be $\Z$-module.
Then the unique group homomorphism
\[ \eta_{M,G}\colon \Hom_\Z(M,\Z)\otimes_\Z G\to \Hom_\Z(M,G) \]
determined by $\eta_{M,G}(\phi\otimes g)(m) = \phi(m)g$ is an $R$-module homomorphism that is functorial in $M$ and $G$.
If $M$ as a $\Z$-module is finitely generated and free, then $\eta_{M,G}$ is an isomorphism.
\end{lemma}
\begin{proof}
Clearly, $\eta_{M,G}$ is a well-defined $\Z$-module homomorphism that is functorial in $M$ and~$G$.
The $R$-module structures on the domain and codomain of $\eta_{M,G}$ are given by Definitions~\ref{def-functorialiso1} and~\ref{def-functorialiso2}.  It is an easy exercise to verify that also with respect to the $R$-module structure, $\eta_{M,G}$ is a homomorphism and  functorial in $M$ and $G$.

Assume that $M$ is finitely generated as a $\Z$-module.  It suffices to show that $\eta_{M,G}$ is an isomorphism of $\Z$-modules, so we may assume that $M=\Z^n$.  For $n=1$, $\eta_{\Z,G}$ is well-known to be an isomorphism.  Since the functors $\Hom_\Z(-,\Z)\otimes_\Z G$ and $\Hom_\Z(-,G)$ are half-exact, they are also split-exact.  By considering the split-exact sequence $\Z^{n-1}\into \Z^n \onto\Z$, an application of the Five Lemma yields that $\eta_{\Z^n,G}$ is an isomorphism if $\eta_{\Z^{n-1},G}$ is. 
\end{proof}

\section{Real rank zero algebras over the pseudocircle $\PS$}
In this section we will show that stable, real rank zero \KSas in the bootstrap class $\mathcal{B}(\PS)$ are strongly classified by $\FKS$, i.e., the missing part of Theorem~\ref{thm-fullFK}. 
We will first need some preliminary homological algebra.




The following result is a direct consequence of 
\cite[Corollary~2.5]{rasmus-and-ralf} as $(\A,\id_{\A})$ and $(\B,\alpha^{-1})$ are liftings of $\FK(\A)$---in the sense of \cite{rasmus-and-ralf}---and the desired equivalence is one implementing the equivalence of the two liftings.
\begin{corollary}
Let \A and \B be separable \cas over $X$ in the bootstrap class $\mathcal B(X)$.
Assume that $\FK(\A)$ consided as an $\NT$-module has projective dimension at most 2, and that $\Ext_{\NT{}}^2(\FK(\A),S\FK(\A))=0$.
Then any isomorphism $\alpha\colon\FK(\A)\to\FK(\B)$ can be lifted to an equivalence in $\KK_X(\A,\B)$.
\end{corollary}

\begin{proposition}[{\cite[Proposition~2]{bentmann:gjogv}}]
\label{prop-fourpoints-dim-most-2}
Let $X$ be a $T_0$-space with at most $4$ points and let \A be a separable \ca over $X$.
Then $\FK(\A)$ has projective dimension at most $2$ in the category of $\NT$-modules. 
\end{proposition}

In the rest of this section, we will consider the connected four point space \PS.  
This space is considered in \cite[Section~6.3]{bentmann:master}, and the filtrated \K-theory is known to coincide with the \cirkt. The $\FKS$ is illustrated by the diagram
\begin{equation}
\vcenter{\xymatrix{    &                         & 13\ar[dr]^{i} \\
3\ar[r]^-{i}\ar[ddr]|!{[dd];[r]}\hole_(.3){i} & 134\ar[ur]^r\ar[r]_r\ar[dr]_i & 14\ar[ddr]|!{[dd];[r]}\hole^(.7){i} & 123 \ar[r]^r\ar[ddr]|!{[dd];[r]}\hole^(.7)r & 1\ar[r]|-\circ^-{\delta}\ar[ddr]|!{[dd];[r]}\hole^(.7){\delta}|(.7)\circ & 3 \\ 
                &                         & 1234\ar[ur]|!{[u];[rd]}\hole^(.7)r\ar[dr]|!{[d];[ru]}\hole_(.7)r \\
4\ar[r]_-{i}\ar[uur]^(.3){i} & 234\ar[dr]_r\ar[r]_r\ar[ur]^i & 23\ar[uur]_(.7){i} & 124 \ar[r]_r\ar[uur]_(.7)r & 2\ar[r]|-\circ_-{\delta}\ar[uur]_(.7){\delta}|(.7)\circ & 4 \\
                &                         & 24\ar[ur]_i
}}\label{eq-diagram-pseudocircle}
\end{equation}
Unlike the other four-point spaces, there is not (yet) a refined invariant for which we have a general UCT over $\PS$. Therefore we need to use other techniques here.

Assume that $M=\FKS(\A)$ for some separable \ca over $\PS$, and consider this as an $\NT{}$-module.
By reducing the problem in several steps, we want to show that $\Ext_\NT^2(M,SM)=0$ when $M$ is \rrzero{}.

\begin{definition}[{\cite{MR2953205}, \cite{arXiv:1301.7223v2}}]
Let $X$ be a finite $T_0$-space.
An $\NT$-module $M$ over $X$ is called \emph{exact} if for all $Y\in\LC(X)$ and all $U\in\Op(Y)$, the sequence
\[ \xymatrix{
M(U_0) \ar[r]^-{i_{U_0}^{Y_0}} & M(Y_0) \ar[r]^-{r_{Y_0}^{Y\setminus U_0}} & M(Y\setminus U_0) \ar[d]^-{\delta_{Y\setminus U_0}^{U_1}} \\
M(Y\setminus U_1) \ar[u]^-{\delta_{Y\setminus U_1}^{U_0}} & M(Y_1)\ar[l]^-{r_{Y_1}^{Y\setminus U_1}} & M(U_1) \ar[l]^-{i_{U_1}^{Y_1}}
} \]
is exact.
An exact $\NT$-module $M$ over $X$ is called \emph{\rrzero{}} if ${\delta_{Y\setminus U_0}^{U_1}}$ vanishes for all $Y\in\LC(X)$ and all $U\in\Op(Y)$.
\end{definition}
Clearly, the $\NT$-module $\FK(A)$ is exact for any \ca{} \A over a finite $T_0$-space $X$.
If furthermore the \ca{} \A has real rank zero, then $\FK(\A)$ is \rrzero{}, and if \A is a \KXa{} then \A has real rank zero if and only if $\FK(\A)$ is \rrzero{}, cf.~\cite[Remark~3.9]{arXiv:1301.7223v2}.

\begin{lemma}\label{lem-two-out-of-three}
Let $L\into M\onto N$ be a short exact sequence of $\NT{}$-modules over a finite $T_0$-space $X$. 
If two of the modules $L$, $M$ and $N$ are exact, so is the third. 
If $M$ is \rrzero{}, so are $L$ and $N$. 
\end{lemma}
\begin{proof}
First part is \cite[Proposition~3.9]{MR2953205}. Second part is clear.  
\end{proof}

\begin{lemma}\label{lem-ext-2-is-zero-for-reduced}
Let $M$ and $N$ be exact $\NT{}$-modules over $\PS$ with 
$M(3_1)$, $M(4_1)$, $M(134_1)$, $M(234_1)$,
$M(1_0)$, $M(2_0)$, $M(123_0)$, $M(124_0)$,
$N(1_0)$, $N(2_0)$, $N(123_0)$, and $N(124_0)$ being zero. 
Then $\Ext_\NT^2(M,SN)=0$. 
\end{lemma}
\begin{proof}
We need to make a projective resolution of $M$. First we examine the special structure of $M$.
For $i=1,2$ and $j=3,4$, we have cyclic six term exact sequences 
$$\xymatrix{
M(j_0)\ar[r] & M(ij_0)\ar[r] & M(i_0)\ar[d] \\ 
M(i_1)\ar[u] & M(ij_1)\ar[l] & M(j_1)\ar[l] 
}$$
and because of the zeros in $M$, we have surjective homomorphisms $M(j_0)\rightarrow M(ij_0)$. 
Similarly, we can show that we have surjective homomorphisms $M(3_0)\oplus M(4_0)\rightarrow M(1234_0)$, $M(3_0)\oplus M(4_0)\rightarrow M(134_0)$, $M(3_0)\oplus M(4_0)\rightarrow M(234_0)$, $M(1_1)\oplus M(2_1)\rightarrow M(3_0)$ and $M(1_1)\oplus M(2_1)\rightarrow M(4_0)$. 
Thus if we have a $\NT{}$-module homomorphism from an $\NT{}$-module $L$ to $M$ that is surjective when restricted to $L(1_1)\rightarrow M(1_1)$ and $L(2_1)\rightarrow M(2_1)$, it will automatically be surjective on the whole of the even degree. 
Because of its special structure, we thus can get a projective object of the form 
\begin{align*}
P
&=%
\bigoplus_{i\in I_{13}}SP_{13}\oplus\bigoplus_{i\in I_{14}}SP_{14}
\oplus\bigoplus_{i\in I_{1234}}SP_{1234}
\oplus\bigoplus_{i\in I_{23}}SP_{23}\oplus\bigoplus_{i\in I_{24}}SP_{24} \\
&\quad\oplus\bigoplus_{i\in I_{123}}SP_{123}\oplus\bigoplus_{i\in I_{124}}SP_{124}
\oplus\bigoplus_{i\in I_{1}}SP_{1}\oplus\bigoplus_{i\in I_{2}}SP_{2}
\end{align*}
and an epimorphism $\phi$ from $P$ to $M$. 

The kernel $\widetilde{P}$ of $\phi$ satisfies, that $\widetilde{P}(Y)$ is a free abelian group for all $Y\in\LC(\PS)^*$. 
Moreover, since $P(3_1)$, $P(4_1)$, $P(134_1)$, $P(234_1)$, $P(123_0)$, $P(124_0)$, $P(1_0)$, and $P(2_0)$ are zero, the same holds for $\widetilde{P}$. 
Consequently, we can get a projective object of the form 
\begin{align*}
P'
&=\bigoplus_{i\in I_{13}'}SP_{13}\oplus\bigoplus_{i\in I_{14}'}SP_{14}
\oplus\bigoplus_{i\in I_{1234}'}SP_{1234}
\oplus\bigoplus_{i\in I_{23}'}SP_{23}\oplus\bigoplus_{i\in I_{24}'}SP_{24} \\
&\quad\oplus\bigoplus_{i\in I_{123}'}SP_{123}\oplus\bigoplus_{i\in I_{124}'}SP_{124}
\oplus\bigoplus_{i\in I_{1}'}SP_{1}\oplus\bigoplus_{i\in I_{2}'}SP_{2}
\end{align*}
and an epimorphism $\phi'$ from $P'$ to $\widetilde{P}$ such that the homomorphisms 
$P'(13_1)\rightarrow\widetilde{P}(13_1)$, 
$P'(14_1)\rightarrow\widetilde{P}(14_1)$, 
$P'(1234_1)\rightarrow\widetilde{P}(1234_1)$, 
$P'(23_1)\rightarrow\widetilde{P}(23_1)$, and 
$P'(24_1)\rightarrow\widetilde{P}(24_1)$
are isomorphisms. 
The kernel $\widetilde{P'}$ of $\phi'$ satisfies, that $\widetilde{P'}(Y)$ is a free abelian group for all $Y\in\LC(\PS)^*$. 
Moreover, $\widetilde{P'}(3_1)$, $\widetilde{P'}(4_1)$, $\widetilde{P'}(134_1)$, $\widetilde{P'}(234_1)$, $\widetilde{P'}(13_1)$, $\widetilde{P'}(14_1)$, $\widetilde{P'}(23_1)$, $\widetilde{P'}(24_1)$, $\widetilde{P'}(1234_1)$, $\widetilde{P'}(123_0)$, $\widetilde{P'}(124_0)$, $\widetilde{P'}(1_0)$, and $\widetilde{P'}(2_0)$ are zero. 
We now construct 
a projective object of the form 
\begin{align*}
P''
&=\bigoplus_{i\in I_{123}''}SP_{123}\oplus\bigoplus_{i\in I_{124}''}SP_{124}
\oplus\bigoplus_{i\in I_{1}''}SP_{1}\oplus\bigoplus_{i\in I_{2}''}SP_{2}
\end{align*}
and an epimorphism $\phi''$ from $P''$ to $\widetilde{P'}$. 
Let $P'''$ denote the kernel of $\phi''$. 
From Proposition~\ref{prop-fourpoints-dim-most-2} it follows that $M$ has projective dimension at most $2$, and thus $P'''$ is actually a projective object, and we have a projective resolution:
$$0\rightarrow P'''\rightarrow P'' \rightarrow P'\rightarrow P\rightarrow M\rightarrow 0.$$

From Remark~\ref{remark-proj-homo} it follows that $\Hom_\NT{}(P'',SN)$ is completely determined by the homomorphisms 
$P''(123_1)\rightarrow N(123_0)$,
$P''(124_1)\rightarrow N(124_0)$,
$P''(1_1)\rightarrow N(1_0)$, and 
$P''(2_1)\rightarrow N(2_0)$. 
But since these codomains are all zero, 
$\Hom_\NT{}(P'',SN)$ is zero. 
From the definition of $\Ext_\NT^2(M,SN)$ as the second cohomology group of $\Hom_\NT{}(-,SN)$ applied to the above projective resolution, it follows that $\Ext_\NT^2(M,SN)$ is zero.
\end{proof}

\begin{lemma}\label{l:exact}
Let $e$ be an idempotent in the ring $\NT{}$ over a finite $T_0$-space $X$ and let $G$ be an abelian group.  If $\NT{}e$ is a torsion-free abelian group, then $\NT{}e \otimes_{\Z} G$ is an exact $\NT$-module.  If $e\NT$ is a free abelian group, then $\Hom_\Z(e \NT{},G)$ is an exact $\NT{}$-module.
\end{lemma}

\begin{proof}
If $G$ is a projective $\Z$-module, then $\NT{}e \otimes_{\Z} G$ is exact since $- \otimes_{\Z} G$ is an exact functor.  If $G$ is injective, then $\Hom_\Z(e \NT{}, G)$ is exact since $\Hom_\Z( - ,G)$ is an exact functor.  

Using this, we now prove the general case.  Let $P_1\into P_0\onto G$ be a projective resolution of $G$, and let $G\into I_0\onto I_1$ be an injective resolution of $G$.

If the group $\NT e$ is torsion-free, then $\NT e \otimes_\Z -$ is an exact functor of $\Z$-modules.  As in the proof of Proposition~\ref{prop-proj-inj-dim}, we see that
\[ \NT{}e \otimes_{\Z} P_{0} \into \NT{}e \otimes_{\Z} P_{1} \onto \NT{}e \otimes_{\Z} G \]
is an exact sequence of $\NT$-modules.  By Lemma~\ref{lem-two-out-of-three}, we have that $\NT e \otimes_\Z G$ is an exact $\NT$-module.

If the group $e\NT$ is a free abelian group, then $\Hom_\Z(e\NT, -)$  is an exact functor of $\Z$-modules.  So similarly we conclude that
\[ \Hom_\Z(e \NT{},G) \into \Hom_\Z(e \NT{} , I_{0} ) \onto \Hom_\Z(e \NT{} , I_{1} ) \]
is an exact sequence of $\NT$-modules and thereby that $\Hom_\Z(e\NT,G)$ is an exact $\NT$-module.
\end{proof}

\begin{observation} \label{obs:hom-tensor-identical}
We will be considering $\NT$-modules over $\PS$ of the forms $S^iP_Y\otimes_\Z G$ and $\Hom_Z(S^jQ_Z, G)$ for certain $Y$ and $Z$.  We note the following about some of the latter form.
Let $M$ be a right $\NT$-module over $\PS$, and let $N=\Hom_\Z(M,\Z)$ be its dual.
For $Y,Z\in\LC(\PS)$ and $\tau\in\NT(Y,Z)$,
the map $N(Y)\xrightarrow{N(\tau)} N(Z)$ is defined as $\Hom_\Z(M(Y),\Z)\xrightarrow{\Hom_\Z(M(\tau),\Z)}\Hom_\Z(M(Z),\Z)$
induced by $M(Y)\xleftarrow{M(\tau)}M(Z)$. 
Using that $S^iQ_V(Y)$ is finitely generated and free for all $Y,V\in
\LC(\PS)$ to compute $\Hom_\Z(S^iQ_V,\Z)$, we observe the following symmetry over $\PS$:
\begin{align*}
 S^iP_j &\cong \Hom_\Z(S^iQ_{12j},\Z)\\
 S^iP_{k34} &\cong \Hom_\Z(S^iQ_k,\Z)
\end{align*}
as $\NT$-modules for all $i\in\{0,1\}$, $j\in\{3,4\}$, and $k\in\{1,2\}$.
\end{observation}

\begin{proposition}\label{prop:ext-vanishes}
Let $M$ be an exact, \rrzero{} $\NT$-module over $\PS$.
Then $\Ext_\NT^2(M,SM)=0$. 
\end{proposition}
\begin{proof}
We proceed in several steps to reduce to a case covered by Lemma~\ref{lem-ext-2-is-zero-for-reduced}.

First we show that is suffices to establish that $\Ext_\NT^2(M,SM)=0$ for all exact, \rrzero{} $\NT$-modules $M$ with $M(3_1)=0$.
Define $K=SP_3\otimes_\Z M(3_1)$ and note that by Lemma~\ref{l:exact}, $K$ is an exact $\NT$-module since $SP_3=\NT\bar e_3$.   Furthermore, $K$ is isomorphic to $\Hom_\Z(SQ_{123},M(3_1))$ by Lemma~\ref{lemma:hom-tensor-iso} and Observation~\ref{obs:hom-tensor-identical}.  So by Proposition~\ref{prop-proj-inj-dim}, $K$ has projective and injective dimension at most 1.
By Remark~\ref{remark-proj-homo}, we may uniquely extend the identity map on $M(3_1)$ to an $\NT$-homomorphism $K\to M$.
We now prove that this map $K\to M$ is a monomorphism, i.e., that each of the maps are injective when we look at the diagram~\eqref{eq-diagram-pseudocircle}. 
By construction, it is enought to check that the maps from $M(3_1)$ to $M(Y_1)$ are injective for all $Y\in\{134,234,1234,13,23,123\}$.  Since $M$ is \rrzero{} and exact, this is evident using the exactness of  six-term sequences in $M$ where the maps occur.
Since $M$ and $K$ are exact and \rrzero{}, we have a short exact sequence 
$K\into M\onto L$ where $L$ is exact and \rrzero{} and $L(3_1)=0$ (cf.\ Lemma~\ref{lem-two-out-of-three}).  By Lemma~\ref{lem-reduction-ext-2}, $\Ext_\NT^2(M,SM)=0$ if and only if $\Ext_\NT^2(L,SL)=0$.

A symmetrical argument can be carried through for $4_1$.
Therefore, it suffices to establish $\Ext_\NT^2(M,SM)=0$ for all exact $\NT$-modules $M$ with $M(3_1)=0$ and $M(4_1)=0$.

Assume that $M$ is exact with $M(3_1)=0$ and $M(4_1)=0$, and consider now the exact $\NT$-module $K=SP_{134}\otimes_\Z M(134_1)$. 
By Lemma~\ref{lemma:hom-tensor-iso} and Observation~\ref{obs:hom-tensor-identical}, $K$ is isomorphic to $\Hom_\Z(SQ_{1},M(134_1))$.  So by Proposition~\ref{prop-proj-inj-dim}, $K$ has projective and injective dimension at most 1.
Extend the identity on $M(134_1)$ uniquely to an $\NT$-homomorphism $K\to M$ (cf.\ Remark~\ref{remark-proj-homo}).
We note that the map $K\to M$ is a monomorphism, i.e., that each of the maps are injective when we look at the diagram~\eqref{eq-diagram-pseudocircle}. 
By construction,
it is enough to check that the maps from $M(134_1)$ to $M(Y_1)$ are injective for all $Y\in\{1234,13,14,123,124,1\}$.
Using that $M$ is exact with $M(3_1)=0$ and $M(4_1)=0$, it is easy to check for all these maps but $M(134_1)\rightarrow M(123_1)$ and $M(134_1)\rightarrow M(124_1)$ that they are injective.
We show that $M(134_1)\rightarrow M(123_1)$ is injective, the other case is analogous. 
First note that the commutativity relations say that the map $M(134_1)\rightarrow M(123_1)$ is equal to both $i_{13_1}^{123_1}r_{134_1}^{13_1}$ and $r_{1234_1}^{123_1}i_{134_1}^{1234_1}$. 
We have already seen that $r_{134_1}^{13_1}$ is injective, and injectivity of $i_{13_1}^{123_1}$ is evident using exactness of $M$ and that $M(3_1)=0$. 
Since $M$ and $K$ are exact and \rrzero{}, we have a short exact sequence 
$K\into M\onto L$ where $L$ is exact and \rrzero{} (cf.\ Lemma~\ref{lem-two-out-of-three}). Note that $L(3_1)$, $L(4_1)$, and $L(134_1)$ vanish.  By Lemma~\ref{lem-reduction-ext-2}, $\Ext_\NT^2(M,SM)=0$ if and only if $\Ext_\NT^2(L,SL)=0$.

A symmetrical argument can be carried through for $234_1$.
Therefore, it suffices to establish $\Ext_\NT^2(M,SM)=0$ for all exact $\NT$-modules $M$ with $M(3_1)$, $M(4_1)$, $M(134_1)$, and $M(234_1)$  vanishing.

Assume that $M$ is exact with with $M(3_1)$, $M(4_1)$, $M(134_1)$, and $M(234_1)$  vanishing.
Define $L=\Hom_\Z(Q_1, M(1_0))$ and note that by Lemma~\ref{l:exact}, $L$ is exact.
Furthermore, $L$ is isomorphic to $P_{134}\otimes_\Z M(1_0)$ by Lemma~\ref{lemma:hom-tensor-iso} and Observation~\ref{obs:hom-tensor-identical}.  So by Proposition~\ref{prop-proj-inj-dim}, $L$ has injective and projective dimension at most 1.
By Remark~\ref{remark-inj-homo}, we may uniquely extend the identity on $M(1_0)$ to an $\NT$-homomorphism $M\to L$.
We now prove that this map $M\to L$ is an epimorphism, i.e., that each of the maps are surjective when we look at the diagram~\eqref{eq-diagram-pseudocircle}. 
By construction, 
it is enough to check that the maps from $M(Y_0)$ to $M(1_0)$ are surjective for all $Y\in\{134,1234,13,14,123,124\}$. 
Because $M$ is exact and has certain groups vanishing, this is evident from the exactness of certain six-term sequences.
Since $M$ and $L$ are exact and \rrzero{}, we have a short exact sequence 
$K\into M\onto L$ where $K$ is exact and \rrzero{} (cf.\ Lemma~\ref{lem-two-out-of-three}). Note that $K(3_1)$, $K(4_1)$, $K(134_1)$, $K(234_1)$, and $K(1_0)$ vanish.  By Lemma~\ref{lem-reduction-ext-2}, $\Ext_\NT^2(M,SM)=0$ if and only if $\Ext_\NT^2(K,SK)=0$.

A symmetrical argument can be carried through for $2_0$.
Therefore, it suffices to establish $\Ext_\NT^2(M,SM)=0$ for all exact $\NT$-modules $M$ with $M(3_1)$, $M(4_1)$, $M(134_1)$, $M(234_1)$, $M(1_0)$, and $M(2_0)$  vanishing.

Assume that $M$ is exact with $M(3_1)$, $M(4_1)$, $M(134_1)$, $M(234_1)$, $M(1_0)$, and $M(2_0)$  vanishing.
Consider the exact $\NT$-module $L=\Hom_\Z(Q_{123},M(123_0))$, and note that it is isomorhic to $P_{3}\otimes_\Z M(123_0)$ by Lemma~\ref{lemma:hom-tensor-iso} and Observation~\ref{obs:hom-tensor-identical}.
By Proposition~\ref{prop-proj-inj-dim}, $L$ has injective and projective dimension at most 1.
Define an $\NT$-homomorphism $M\to L$ by uniquely extending the identity on $M(123_0)$ using Remark~\ref{remark-inj-homo}.
We now prove that this map $M\to L$ is an epimorphism, i.e., that each of the maps are surjective when we look at the diagram~\eqref{eq-diagram-pseudocircle}. 
By construction, 
it is enough to check that the maps from $M(Y_0)$ to $M(123_0)$ are surjective for all $Y\in\{3,134,234,1234,13,23\}$. 
But since $M$ is exact and has certain groups vanishing, this is evident from certain six-term exact sequences for all maps but the maps $M(134_0)\rightarrow M(123_0)$ and $M(134_0)\rightarrow M(124_0)$. 
We show it for the former, the latter is analogous. 
First note that the commutativity relations say that the map $M(134_0)\rightarrow M(123_0)$ is equal to both $i_{13_0}^{123_0}r_{134_0}^{13_0}$ and $r_{1234_0}^{123_0}i_{134_0}^{1234_0}$. 
We have already seen that $i_{13_0}^{123_0}$ is surjective, and surjectivity of $r_{134_0}^{13_0}$ is evident from a six-term exact sequence in $M$ using that $M(4_1)=0$. 
Since $M$ and $L$ are exact and \rrzero{}, we have a short exact sequence 
$K\into M\onto L$ where $K$ is exact and \rrzero{} (cf.\ Lemma~\ref{lem-two-out-of-three}).  Note that $K(3_1)$, $K(4_1)$, $K(134_1)$, $K(234_1)$, $K(1_0)$, $K(2_0)$ and $K(123_0)$ vanish.  By Lemma~\ref{lem-reduction-ext-2}, $\Ext_\NT^2(M,SM)=0$ if and only if $\Ext_\NT^2(K,SK)=0$.

A symmetrical argument can be carried through for $124_0$.
Therefore, it suffices to establish $\Ext_\NT^2(M,SM)=0$ for all exact $\NT$-modules $M$ with $M(3_1)$, $M(4_1)$, $M(134_1)$, $M(234_1)$, $M(1_0)$, $M(2_0)$, $M(123_0)$, and $M(124_0)$  vanishing.  And this follows from Lemma~\ref{lem-ext-2-is-zero-for-reduced}.
\end{proof}

\begin{corollary}
Let \A and \B be real rank zero, separable \cas over $\PS$ in the bootstrap class $\mathcal{B}(\PS)$. 
Assume that there is an isomorphism $\alpha\colon\FKS(\A)\rightarrow\FKS(\B)$. 
Then there exists a $\KK_{\PS}$ equivalence $x$ such that $\FKS(x)=\alpha$. 
\end{corollary}

\begin{corollary}\label{cor-strong-classification-pseudo-circle}
Let \A and \B be real rank zero \KSas in the bootstrap class $\mathcal{B}(\PS)$. 
Assume that there is an isomorphism $\alpha\colon\FKS(\A)\rightarrow\FKS(\B)$. 
Then there exists an $\PS$-equivariant isomorphism $\phi$ such that $\FKS(\phi)=\alpha$. 
\end{corollary}

\section*{Acknowledgement}
This work was supported by the Danish National Research Foundation through the Centre for Symmetry and Deformation (DNRF92) and by a grant from the Simons Foundation (\#279369 to Efren Ruiz).
The second and third named authors also want to thank the Department of Mathematical Sciences at University of Copenhagen for hospitality during visits where this work was carried out. 
Also the authors thank Rasmus Bentmann for sharing early versions of the papers \cite{bentmann:gjogv} and \cite{rasmus-and-ralf}.


\begin{thebibliography}{ERR13b}

\bibitem[ABK13]{arXiv:1301.7223v2}
Sara~E. Arklint, Rasmus Bentmann, and Takeshi Katsura, \emph{Reduction of
  filtered {K}-theory and a characterization of {C}untz-{K}rieger algebras},
  ArXiv e-prints (2013), \href {http://arxiv.org/abs/1301.7223v2}
  {\path{arXiv:1301.7223v2}}.

\bibitem[ABK14]{arXiv:1309.7162v1}
\bysame, \emph{The {K}-theoretical range of {C}untz-{K}rieger algebras}, J.
  Funct. Anal. (2014), to appear, \href {http://arxiv.org/abs/1309.7162v1}
  {\path{arXiv:1309.7162v1}}.

\bibitem[Ark13]{arklint:gjogv}
Sara~E. Arklint, \emph{Do phantom {C}untz-{K}rieger algebras exist?}, Operator
  {A}lgebra and {D}ynamics: {N}ordforsk Network Closing Conference, Faroe
  Islands, May 2012 (Toke~M. Carlsen, S{\o}ren Eilers, Gunnar Restorff, and
  Sergei Silvestrov, eds.), Springer Proceedings in Mathematics \& Statistics,
  vol.~58, Springer, Berlin, 2013, URL:
  \url{http://dx.doi.org/10.1007/978-3-642-39459-1_2}, \href
  {http://dx.doi.org/10.1007/978-3-642-39459-1_2}
  {\path{doi:10.1007/978-3-642-39459-1_2}}.

\bibitem[ARR12]{MR2949216}
Sara Arklint, Gunnar Restorff, and Efren Ruiz, \emph{Filtrated {K}-theory for
  real rank zero {$C^*$}-algebras}, Internat. J. Math. \textbf{23} (2012),
  no.~8, 1250078, 19, URL: \url{http://dx.doi.org/10.1142/S0129167X12500784},
  \href {http://dx.doi.org/10.1142/S0129167X12500784}
  {\path{doi:10.1142/S0129167X12500784}}. \MR{2949216}

\bibitem[Ben]{bentmann}
Rasmus Bentmann, \emph{Kirchberg {$X$}-algebras with real rank zero and
  intermediate cancellation}, J. Noncomm.~Geom., to appear., \href
  {http://arxiv.org/abs/1301.6652} {\path{arXiv:1301.6652}}.

\bibitem[Ben10]{bentmann:master}
\bysame, \emph{Filtrated {$K$}-theory and classification of
  {$C\sp*$}-algebras}, Diploma thesis, Georg-August-Universit\"{a}t
  G\"{o}ttingen, 2010, URL:
  \url{http://www.uni-math.gwdg.de/rbentma/diplom_thesis.pdf}.

\bibitem[Ben13]{bentmann:gjogv}
\bysame, \emph{Projective dimension in filtrated {$K$}-theory}, Operator
  {A}lgebra and {D}ynamics: {N}ordforsk Network Closing Conference, Faroe
  Islands, May 2012 (Toke~M. Carlsen, S{\o}ren Eilers, Gunnar Restorff, and
  Sergei Silvestrov, eds.), Springer Proceedings in Mathematics \& Statistics,
  vol.~58, Springer, Berlin, 2013, URL:
  \url{http://dx.doi.org/10.1007/978-3-642-39459-1_3}, \href
  {http://dx.doi.org/10.1007/978-3-642-39459-1_3}
  {\path{doi:10.1007/978-3-642-39459-1_3}}.

\bibitem[BH03]{MR1990568}
Mike Boyle and Danrun Huang, \emph{Poset block equivalence of integral
  matrices}, Trans. Amer. Math. Soc. \textbf{355} (2003), no.~10, 3861--3886
  (electronic), URL: \url{http://dx.doi.org/10.1090/S0002-9947-03-02947-7},
  \href {http://dx.doi.org/10.1090/S0002-9947-03-02947-7}
  {\path{doi:10.1090/S0002-9947-03-02947-7}}. \MR{1990568 (2004f:15020)}

\bibitem[BK11]{arXiv:1101.5702v3}
Rasmus {B}entmann and Manuel {K}\"ohler, \emph{Universal coefficient theorems
  for {$C\sp*$}-algebras over finite topological spaces}, ArXiv e-prints
  (2011), \href {http://arxiv.org/abs/1101.5702v3} {\path{arXiv:1101.5702v3}}.

\bibitem[BM]{rasmus-and-ralf}
Rasmus Bentmann and Ralf Meyer, \emph{A more general method to classify up to
  equivariant {$KK$}-equivalence}, in preparation.

\bibitem[Bon04]{bonkat:phd}
Alexander Bonkat, \emph{Bivariante {$K$}-{T}heorie f\"ur {K}ategorien
  projektiver {S}ysteme von {$C\sp*$}-{A}lgebren}, Ph.d. thesis,
  Westf\"alischen Wilhelms-Universit\"at M\"unster, jan 2004, Published in
  Preprintreihe SFB 478 Geometrische Strukturen in der Mathematik, Heft 319,
  URL: \url{http://wwwmath.uni-muenster.de/sfb/about/publ/heft319.ps}.

\bibitem[CK80]{CK80}
Joachim Cuntz and Wolfgang Krieger, \emph{A class of {$C^{\ast} $}-algebras and
  topological {M}arkov chains}, Invent. Math. \textbf{56} (1980), no.~3,
  251--268, URL: \url{http://dx.doi.org/10.1007/BF01390048}, \href
  {http://dx.doi.org/10.1007/BF01390048} {\path{doi:10.1007/BF01390048}}.
  \MR{561974 (82f:46073a)}

\bibitem[ER06]{MR2265044}
S{\o}ren Eilers and Gunnar Restorff, \emph{On {R}\o rdam's classification of
  certain {$C^*$}-algebras with one non-trivial ideal}, Operator {A}lgebras:
  {T}he {A}bel {S}ymposium 2004, Abel Symp., vol.~1, Springer, Berlin, 2006,
  pp.~87--96, URL: \url{http://dx.doi.org/10.1007/978-3-540-34197-0_4}, \href
  {http://dx.doi.org/10.1007/978-3-540-34197-0_4}
  {\path{doi:10.1007/978-3-540-34197-0_4}}. \MR{2265044 (2007m:46098)}

\bibitem[ERR13a]{ERR:gjogv}
S{\o}ren Eilers, Gunnar Restorff, and Efren Ruiz, \emph{Classification of graph
  {$C\sp*$}-algebras with no more than four primitive ideals}, Operator
  {A}lgebra and {D}ynamics: {N}ordforsk Network Closing Conference, Faroe
  Islands, May 2012 (Toke~M. Carlsen, S{\o}ren Eilers, Gunnar Restorff, and
  Sergei Silvestrov, eds.), Springer Proceedings in Mathematics \& Statistics,
  vol.~58, Springer, Berlin, 2013, URL:
  \url{http://dx.doi.org/10.1007/978-3-642-39459-1_5}, \href
  {http://dx.doi.org/10.1007/978-3-642-39459-1_5}
  {\path{doi:10.1007/978-3-642-39459-1_5}}.

\bibitem[ERR13b]{arXiv:1301.7695v1}
S\o{}ren Eilers, Gunnar Restorff, and Efren Ruiz, \emph{Strong classification
  of extensions of classifiable {$C\sp*$}-algebras}, ArXiv e-prints (2013),
  \href {http://arxiv.org/abs/1301.7695v1} {\path{arXiv:1301.7695v1}}.

\bibitem[HS97]{MR1438546}
P.~J. Hilton and U.~Stammbach, \emph{A course in homological algebra}, second
  ed., Graduate Texts in Mathematics, vol.~4, Springer-Verlag, New York, 1997,
  URL: \url{http://dx.doi.org/10.1007/978-1-4419-8566-8}, \href
  {http://dx.doi.org/10.1007/978-1-4419-8566-8}
  {\path{doi:10.1007/978-1-4419-8566-8}}. \MR{1438546 (97k:18001)}

\bibitem[Kir94]{kirchberg1994}
Eberhard Kirchberg, \emph{The classification of purely infinite of
  {$C^*$}-algebras using {K}asparov's theory}, preprint, 1994.

\bibitem[Kir00]{MR1796912}
\bysame, \emph{Das nicht-kommutative {M}ichael-{A}uswahlprinzip und die
  {K}lassifikation nicht-einfacher {A}lgebren}, {$C^*$}-algebras ({M}\"unster,
  1999), Springer, Berlin, 2000, pp.~92--141. \MR{1796912 (2001m:46161)}

\bibitem[MN09]{MR2545613}
Ralf Meyer and Ryszard Nest, \emph{{$C^*$}-algebras over topological spaces:
  the bootstrap class}, M\"unster J. Math. \textbf{2} (2009), 215--252.
  \MR{2545613 (2011a:46105)}

\bibitem[MN12]{MR2953205}
\bysame, \emph{{${\rm C}^*$}-algebras over topological spaces: filtrated
  {K}-theory}, Canad. J. Math. \textbf{64} (2012), no.~2, 368--408, URL:
  \url{http://dx.doi.org/10.4153/CJM-2011-061-x}, \href
  {http://dx.doi.org/10.4153/CJM-2011-061-x}
  {\path{doi:10.4153/CJM-2011-061-x}}. \MR{2953205}

\bibitem[Phi00]{phillips2000}
N.~Christopher Phillips, \emph{A classification theorem for nuclear purely
  infinite simple {$C^*$}-algebras}, Doc. Math. \textbf{5} (2000), 49--114
  (electronic). \MR{1745197 (2001d:46086b)}

\bibitem[Res06]{MR2270572}
Gunnar Restorff, \emph{Classification of {C}untz-{K}rieger algebras up to
  stable isomorphism}, J. Reine Angew. Math. \textbf{598} (2006), 185--210,
  URL: \url{http://dx.doi.org/10.1515/CRELLE.2006.074}, \href
  {http://dx.doi.org/10.1515/CRELLE.2006.074}
  {\path{doi:10.1515/CRELLE.2006.074}}. \MR{2270572 (2007m:46090)}

\bibitem[Res08]{restorff:phd}
\bysame, \emph{Classification of non-simple {$C\sp*$}-algebras}, Ph.d. thesis,
  University of Copenhagen, 2008, URL:
  \url{http://www.math.ku.dk/~restorff/papers/thesis.pdf}.

\bibitem[R{\o}r95]{MR1340839}
Mikael R{\o}rdam, \emph{Classification of {C}untz-{K}rieger algebras},
  $K$-Theory \textbf{9} (1995), no.~1, 31--58, URL:
  \url{http://dx.doi.org/10.1007/BF00965458}, \href
  {http://dx.doi.org/10.1007/BF00965458} {\path{doi:10.1007/BF00965458}}.
  \MR{1340839 (96k:46103)}

\bibitem[RR07]{MR2379290}
Gunnar Restorff and Efren Ruiz, \emph{On {R}\o rdam's classification of certain
  {$C^*$}-algebras with one non-trivial ideal. {II}}, Math. Scand. \textbf{101}
  (2007), no.~2, 280--292. \MR{2379290 (2009h:46120)}

\end{thebibliography}

\providecommand{\bysame}{\leavevmode\hbox to3em{\hrulefill}\thinspace}
\providecommand{\MR}{\relax\ifhmode\unskip\space\fi MR }
\providecommand{\MRhref}[2]{%
  \href{http://www.ams.org/mathscinet-getitem?mr=#1}{#2}
}
\providecommand{\href}[2]{#2}

\end{document}